\documentclass[12pt]{amsart}
\usepackage{amscd,amsmath,amsthm,amssymb}
\usepackage{mathtools}
\usepackage[margin=2cm]{geometry}
\usepackage{epsfig}
\usepackage{rawfonts}
\usepackage{enumerate}
\usepackage{graphics}
\usepackage{multirow}
\usepackage{xspace}
\usepackage{graphicx}
\usepackage[normalem]{ulem}
\usepackage{mathrsfs}
\usepackage{amsmath}
\usepackage{amsfonts}
\usepackage{amssymb}
\usepackage{amsthm}
\usepackage{graphicx}
\usepackage{booktabs}
\usepackage{caption}
\usepackage{listings}
\usepackage{setspace}
\usepackage[mathscr]{eucal}
\usepackage{pgfplots}
\usepackage{hyperref}
\usepackage{wrapfig}
\usepackage{floatflt,epsfig}
\usepackage{ dsfont }
\usepackage{amscd}
\usepackage{tikz-cd}
\usepackage{fancyhdr}
\usepackage[all]{xy}
\usepackage{latexsym}
\usepackage{amscd}
\usepackage{pifont}
\usepackage{subfig}
\usepackage{easyReview}
\usepackage{subfig}
\usepackage{pstricks-add}
\usepackage{pgf,tikz,pgfplots}
\pgfplotsset{compat=1.15}
\usepackage{mathrsfs}
\usetikzlibrary{arrows}
\usetikzlibrary[patterns]
\usepackage[ruled,vlined]{algorithm2e}
\SetKwInput{KwInput}{Input}
\SetKwInput{KwOutput}{Output}

\usepackage{multicol}

\def\NZQ{\Bbb}

\def\ZZ{{\NZQ Z}}

\newcommand{\Cc}{\mathcal{C}}

\newcommand{\cC}{\mathcal{C}}
\newcommand{\cD}{\mathcal{D}}

\newcommand{\cP}{\mathcal{P}}

\newcommand{\cF}{\mathcal{F}}

\newcommand{\cR}{\mathcal{R}}

\newcommand{\cS}{\mathcal{S}}
\newcommand{\cT}{\mathcal{T}}

\newcommand{\cG}{\mathcal{G}}

\renewcommand{\qedsymbol}{$\square$}

\newcommand{\fp}{\mathfrak{p}}

\def\opn#1#2{\def#1{\operatorname{#2}}} 
\opn\chara{char} \opn\length{\ell} \opn\pd{pd} \opn\rk{rk}
\opn\projdim{proj\,dim} \opn\injdim{inj\,dim} \opn\rank{rank}
\opn\depth{depth} \opn\grade{grade} \opn\height{height}
\opn\embdim{emb\,dim} \opn\codim{codim}

\opn\Tr{Tr} \opn\bigrank{big\,rank}
\opn\superheight{superheight}\opn\lcm{lcm}
\opn\trdeg{tr\,deg}
	\opn\reg{reg} \opn\lreg{lreg} \opn\ini{in} \opn\lpd{lpd}
	\opn\size{size} \opn\sdepth{sdepth}
	\opn\link{link}\opn\fdepth{fdepth}\opn\lex{lex}\opn\dist{dist}
	\opn\div{div} \opn\Div{Div} \opn\cl{cl} \opn\Cl{Cl}
	\opn\Spec{Spec} \opn\Supp{Supp} \opn\supp{supp} \opn\Sing{Sing}
	\opn\Ass{Ass} \opn\Min{Min}\opn\Mon{Mon}
	\opn\Ann{Ann} \opn\Rad{Rad} \opn\Soc{Soc}
	\opn\Im{Im} \opn\Ker{Ker} \opn\Coker{Coker} \opn\Am{Am}
	\opn\Hom{Hom} \opn\Tor{Tor} \opn\Ext{Ext} \opn\End{End}
	\opn\Aut{Aut} \opn\id{id}
	
	\opn\nat{nat}
	\opn\pff{pf}
	\opn\Pf{Pf} \opn\GL{GL} \opn\SL{SL} \opn\mod{mod} \opn\ord{ord}
	\opn\Gin{Gin} \opn\Hilb{Hilb}\opn\sort{sort}
	\opn\aff{aff} \opn
	\con{conv} \opn\relint{relint} \opn\st{st}
	\opn\lk{lk} \opn\cn{cn} \opn\core{core} \opn\vol{vol}
	\opn\link{link} \opn\star{star}\opn\lex{lex}\opn\set{set}
	\opn\gr{gr}
	
	\def\pot#1#2{#1[\kern-0.28ex[#2]\kern-0.28ex]}

	\opn\dirlim{\underrightarrow{\lim}}
	\opn\inivlim{\underleftarrow{\lim}}
	\def\Implies{\ifmmode\Longrightarrow \else
		\unskip${}\Longrightarrow{}$\ignorespaces\fi}
	\def\implies{\ifmmode\Rightarrow \else
		\unskip${}\Rightarrow{}$\ignorespaces\fi}
	\def\iff{\ifmmode\Longleftrightarrow \else
		\unskip${}\Longleftrightarrow{}$\ignorespaces\fi}

	\let\:=\colon
	\let\epsilon\varepsilon
	\let\kappa=\varkappa
	\def\qed{\ifhmode\textqed\fi
		\ifmmode\ifinner\quad\qedsymbol\else\dispqed\fi\fi}
	\def\textqed{\unskip\nobreak\penalty50
		\hskip2em\hbox{}\nobreak\hfil\qedsymbol
		\parfillskip=0pt \finalhyphendemerits=0}
	\def\dispqed{\rlap{\qquad\qedsymbol}}
	
	\opn\dis{dis}
	\def\pnt{{\raise0.5mm\hbox{\large\bf.}}}
	
	\opn\Lex{Lex}
	
    \newtheorem{Theorem}{Theorem}[section]
	\newtheorem{Lemma}[Theorem]{Lemma}
	\newtheorem{Corollary}[Theorem]{Corollary}
	\newtheorem{Proposition}[Theorem]{Proposition}
	\newtheorem{Remark}[Theorem]{Remark}
	
	\newtheorem{Example}[Theorem]{Example}
	
	\newtheorem{Definition}[Theorem]{Definition}

    \newtheorem{Discussion}[Theorem]{Discussion}

\title{Minimal Primes and Radicality of Ideals Generated by Adjacent 2-Minors}

\author[T.~Hibi]{Takayuki Hibi}
\address[Takayuki Hibi]
{Department of Pure and Applied Mathematics, 
Graduate School of Information Science and Technology, 
Osaka University, 
Suita, Osaka 565-0871, Japan}
\email{hibi@math.sci.osaka-u.ac.jp}

\author[F. Navarra]{Francesco Navarra}      
   \address[Francesco Navarra]{Sabanci University, Faculty of Engineering and Natural Sciences, Orta Mahalle, Tuzla 34956, Istanbul, Turkey}	
\email{francesco.navarra@sabanciuniv.edu}

\author[A. A. Qureshi]{Ayesha Asloob Qureshi}      
   \address[Ayesha Asloob Qureshi]{Sabanci University, Faculty of Engineering and Natural Sciences, Orta Mahalle, Tuzla 34956, Istanbul, Turkey}	
\email{aqureshi@sabanciuniv.edu, ayesha.asloob@sabanciuniv.edu}

\author[S.~Saeedi~Madani]{Sara Saeedi Madani}
\address[Sara Saeedi Madani]
{Department of Mathematics and Computer Science, Amirkabir University of Technology, Tehran, Iran, and School of Mathematics, Institute for Research in Fundamental Sciences, Tehran, Iran} 
\email{sarasaeedi@aut.ac.ir, sarasaeedim@gmail.com}


\begin{document}

\maketitle

\begin{abstract}
In this paper, we provide a complete description of the minimal primes of ideals generated by adjacent $2$-minors, in terms of the so-called admissible sets and associated lattice ideals. We prove that for these ideals, the properties of being unmixed, Cohen–Macaulay, level, Gorenstein, and complete intersection are equivalent. Moreover, we give a combinatorial characterization of all convex collections of cells satisfying any of these equivalent properties. Finally, we study the radicality of these ideals and derive necessary combinatorial conditions based on minimal non-radical configurations.
\end{abstract}

\section*{Introduction}

Determinantal ideals constitute a classical and extensively studied subject in commutative algebra and algebraic geometry (\cite{BH, BV, HE}). Among several generalizations of determinantal ideals, a prominent and widely studied class is given by ideals generated by \emph{adjacent $2$--minors}.

Let $X = (x_{ij})_{i=1,\ldots,m \atop  j=1,\ldots,n}$  be an $m \times n$ matrix of indeterminates over a field $K$. An \emph{adjacent $2$--minor} is a binomial of the form $ x_{i,j}x_{i+1,j+1} - x_{i+1,j}x_{i,j+1},$ arising as the determinant of a $2 \times 2$ submatrix of $X$. Ideals generated by adjacent $2$--minors arise naturally in several contexts, including the study of lattice and toric ideals \cite{ES} and applications to algebraic statistics, particularly in the theory of contingency tables, that is, matrices of non-negative integers whose row and column sums, called \emph{marginals}, are fixed, see \cite{DES, HHO}. From a statistical perspective, a fundamental problem is to understand the space of all contingency tables with given marginals and, in particular, to determine whether this space is connected under suitable local transformations. Such transformations, usually called \emph{moves}, are integer-valued modifications of a table that preserve the row and column sums while keeping all entries non-negative. A basic class of moves arises from adjacent $2\times2$ sub-matrices: increasing the two diagonal entries $(i,j)$ and $(i+1,j+1)$ by one and simultaneously decreasing the remaining two entries by one yields a new contingency table with the same row and column sums, provided that all entries remain non–negative. Algebraically, such moves correspond to adjacent $2$--minors, and the ideal $I_{\mathrm{adj}}(X)$ generated by all these binomials encodes the connectivity of the space of contingency tables. According to \cite[Section 4]{DES}, two tables with the same marginals are connected by adjacent $2\times2$ moves if and only if the corresponding binomial belongs to $I_{\mathrm{adj}}(X)$.

Motivated by these connections, the algebraic structure of ideals generated by adjacent $2$--minors, particularly their primary decompositions and minimal primes, has been studied extensively. A complete combinatorial description of the minimal primes of $I_{\mathrm{adj}}(X)$ was obtained by Hoşten and Sullivant~\cite{HSu}. Building on this perspective, Herzog and Hibi~\cite{HH} initiated a systematic investigation of ideals generated by arbitrary sets of adjacent $2$--minors. By interpreting these ideals in terms of collections of cells, they developed a combinatorial framework that enabled the study of primality, radicality, and primary decomposition, especially in the convex case.

In this paper, we substantially extend the line of research initiated in~\cite{HH}. We provide a complete description of the minimal primes of ideals generated by adjacent $2$--minors associated with an arbitrary collection of cells. As a consequence, we prove that every unmixed adjacent $2$--minor ideal is a complete intersection, and we obtain a combinatorial characterization of all convex collections of cells satisfying these equivalent conditions. Furthermore, we investigate the radicality of such ideals and establish several necessary conditions.

The paper is organized as follows. In Section~\ref{Sec: preliminaries}, we introduce the necessary definitions used throughout the paper. Section~\ref{Sec: minimal primes} is devoted to the description, in full generality, of the minimal primes of ideals generated by adjacent $2$--minors associated with collections of cells. This description is formulated in terms of admissible sets and lattice ideals (see Theorem~\ref{Thm: admissible set}). Our results extend previously known descriptions (see ~\cite[Lemma~3.1]{HH}) and play a central role in proving one of our main results in Section~\ref{Section: CI = Unmixed}. 
\begin{Theorem} [Theorem~\ref{Thm: Unmixed = CI}]
Let $\cC$ be a collection of cells, and let $I_{\mathrm{adj}}(\cC)\subset S_{\cC}=K[x_v \mid v \text{ is a vertex of } \cC]$ denote its adjacent $2$--minor ideal.
Then the following conditions are equivalent:
\begin{enumerate}
\item $I_{\mathrm{adj}}(\cC)$ is a complete intersection;
\item $S_\cC/I_{\mathrm{adj}}(\cC)$ is Gorenstein;
\item $S_\cC/I_{\mathrm{adj}}(\cC)$ is level;
\item $S_\cC/I_{\mathrm{adj}}(\cC)$ is Cohen–Macaulay;
\item $I_{\mathrm{adj}}(\cC)$ is unmixed.
\end{enumerate}
Moreover, if $I_{\mathrm{adj}}(\cC)$ satisfies above conditions, then $\mathrm{ht}(I_{\mathrm{adj}}(\cC)) =|\cC|$.
\end{Theorem}

\noindent The reader may refer to \cite{BH} for further details on these properties. From a combinatorial--algebraic perspective, we investigate necessary conditions
on the shape of a collection of cells for the validity of the equivalent properties
listed in Theorem~\ref{Thm: Unmixed = CI}. In
Lemmas~\ref{Coro: necessary condition for CI} and
\ref{Coro: necessary condition for CI with X pento},
we show that to satisfy any of these properties, a collection of cells must avoid both a square tetromino
and an $X$--pentomino, see Figure~\ref{fig: square and X-pento}.
Moreover, when the collection of cells is \emph{convex},
the absence of these two configurations is not only necessary
but also sufficient for the properties in
Theorem~\ref{Thm: Unmixed = CI} to hold.
This yields a complete combinatorial characterization,
which is stated in the following result.

\begin{Theorem}[Theorem \ref{Thm: characterization convex CI}]
Let $\cC$ be a convex collection of cells. Then $I_{\mathrm{adj}}(\cC)$ satisfies the equivalent conditions in Theorem~\ref{Thm: Unmixed = CI} if and only if $\cC$ contains neither a square tetromino nor an $X$-pentomino.
\end{Theorem}

\noindent However, this result does not extend beyond the convex case, as shown by a counterexample in Remark~\ref{Rmk: counter-example convex}. Section~\ref{Sec: Radicality} is devoted to the study of radicality, motivated by the combinatorial interpretation of the primary components given in Corollary~\ref{Coro: primary decomposition} through the description of the minimal primes. From a computational perspective, we employ computer algebra systems (see~\cite{N}) to explore examples and verify conjectural patterns related to radicality. In particular, we prove in Theorem~\ref{Thm: infinity minimally non-radical} that, among all collections of cells of rank greater than~$6$, there always exists a collection of cells that is \emph{minimally non-radical}, that is, not radical but with every weakly connected subcollection being radical (see Definition~\ref{Defn:minimally non-radical}). Finally we provide necessary conditions for radicality in Theorem~\ref{Prop: necessary conditions}.

\section{Preliminaries}\label{Sec: preliminaries}

Let $(i,j),(k,l) \in \mathbb{Z}^2$. We define the partial order $(i,j) \leq (k,l)$ if and only if $i \leq k$ and $j \leq l$. Given $a = (i,j)$ and $b = (k,l)$ in $\mathbb{Z}^2$ with $a \leq b$, the set 
\[
[a,b] = \{(m,n) \in \mathbb{Z}^2 : i \leq m \leq k,\ j \leq n \leq l \}
\]
is called an \textit{interval} of $\mathbb{Z}^2$. If $i < k$ and $j < l$, then the interval $[a,b]$ is said to be \textit{proper}; in this case, we refer to $a$ and $b$ as the \textit{diagonal corners} of $[a,b]$, and to $c = (i,l)$ and $d = (k,j)$ as the \textit{anti-diagonal corners}. If $j = l$ (respectively, $i = k$), then $a$ and $b$ are said to be in a \textit{horizontal} (respectively, \textit{vertical}) \textit{position}. 
A proper interval $C = [a,b]$ such that $b = a + (1,1)$ is called a \textit{cell} of $\mathbb{Z}^2$. The points $a$, $b$, $c$, and $d$ are respectively referred to as the \textit{lower left}, \textit{upper right}, \textit{upper left}, and \textit{lower right} \textit{corners} of $C$. We denote by 
\[
V(C) = \{a,b,c,d\} \quad \text{and} \quad E(C) = \{\{a,c\}, \{c,b\}, \{b,d\}, \{a,d\}\}
\]
the sets of \textit{vertices} and \textit{edges} of $C$, respectively.  

For a non-empty collection of cells $\cC$ in $\mathbb{Z}^2$, we define the sets of vertices and edges of $\cC$, respectively, as
\[
V(\cC) = \bigcup_{C \in \cC} V(C), \qquad 
E(\cC) = \bigcup_{C \in \cC} E(C).
\]
The \textit{rank} of $\cC$, denoted by $\vert \cC \vert$, is the number of cells contained in $\cC$.

Given two distinct cells $C$ and $D$ in $\cC$, a \textit{walk} (respectively, a \textit{weak walk}) from $C$ to $D$ in $\cC$ is a sequence of cells $\cC : C = C_1, \dots, C_m = D$ in $\mathbb{Z}^2$ such that $C_i \cap C_{i+1}$ is a common edge (respectively, a common edge or vertex) of $C_i$ and $C_{i+1}$, for all $i = 1, \dots, m-1$. Moreover, if all cells in the sequence are distinct, i.e., $C_i \neq C_j$ for all $i \neq j$, then $\cC$ is called a \textit{path} (respectively, a \textit{weak path}) from $C$ to $D$.  

Two cells $C$ and $D$ are said to be \textit{connected} (respectively, \textit{weakly connected}) in $\cC$ if there exists a path (respectively, a weak path) from $C$ to $D$ in $\cC$. A \textit{polyomino} $\cP$ is a finite collection of cells in $\mathbb{Z}^2$ such that any two cells in $\cP$ are connected in $\cP$. 
A finite, non-empty collection of cells $\cC$ is said to be \textit{weakly connected} if any two cells in $\cC$ are weakly connected in $\cC$.  

A \textit{connected component} $\cC'$ of $\cC$ is a polyomino contained in $\cC$ that is maximal with respect to inclusion; that is, if $A \in \cC \setminus \cC'$, then $\cC' \cup \{A\}$ is not a polyomino. For instance, the weakly connected collection of cells shown in Figure~\ref{Figure: Example weakly conn. coll. of cells}~(A) consists of four connected components. Observe that every polyomino is a weakly connected collection of cells. Moreover, for convention, we assume that the empty-set is also a collection of cells.

\begin{figure}[h]
 	\centering
 	\subfloat[]{\includegraphics[scale=0.6]{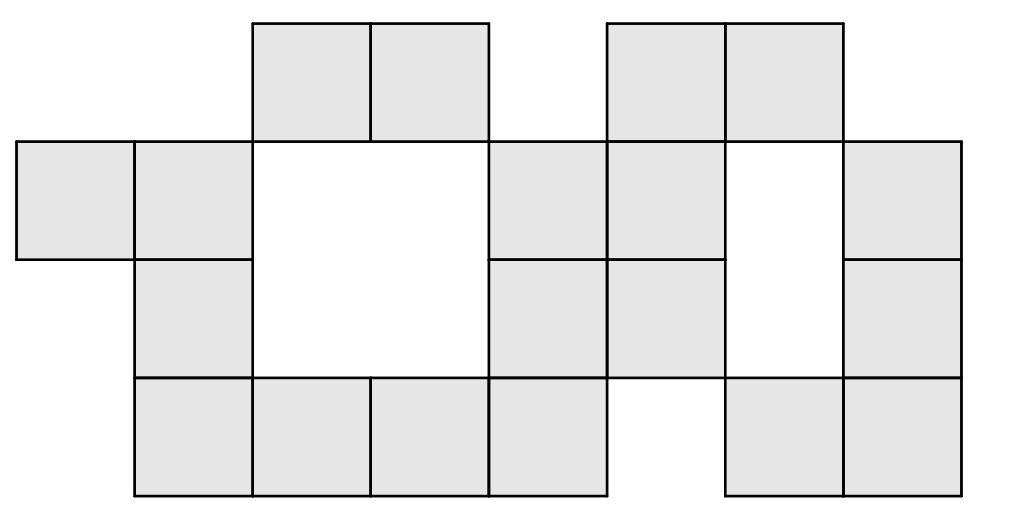}}\qquad\qquad
	\subfloat[]{\includegraphics[scale=0.8]{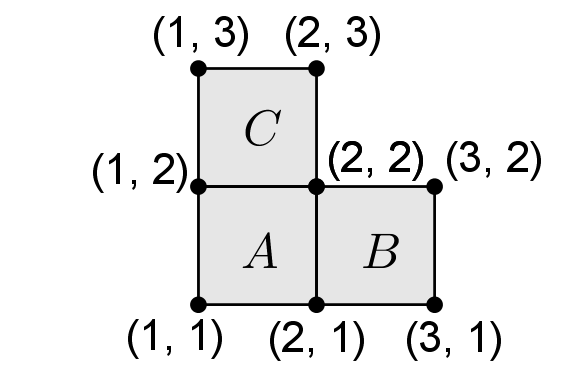}}
    \caption{A weakly connected collection of cells and a polyomino.}
	\label{Figure: Example weakly conn. coll. of cells}
\end{figure}

Consider two cells $A$ and $B$ of $\mathbb{Z}^2$ with $a=(i,j)$ and $b=(k,l)$ as the lower left corners of $A$ and $B$, respectively, with $a\leq b$. The \textit{cell interval} $[A,B]$, also called a \textit{rectangle}, is the set of cells of $\mathbb{Z}^2$ whose lower left corner $(r,s)$ satisfies $i\leqslant r\leqslant k$ and $j\leqslant s\leqslant l$.

If $\cC$ is a collection of cells, then a \textit{maximal rectangle} of $\cC$ is a rectangle $\cR$ of cells of $\cC$ such that no other rectangle whose cells all belong to $\cC$ properly contains $\cR$. Moreover, an interval $[a,b]$ of $\ZZ^2$ is called the \textit{minimal bounding box} of $\cC$ if it is the smallest interval of $\ZZ^2$ with respect to the set inclusion that contains $V(\cC)$.

Having established the preceding combinatorial definitions and notation, we now turn to Commutative Algebra. In \cite{HH} and \cite{Q}, the authors introduced the notions of the \textit{adjacent 2-minor ideal} and the \textit{inner 2-minor ideal} of a collection of cells, which we briefly recall here.

Let $\cC$ be a finite and non-empty collection of cells, and let $S_\cC = K[x_v \mid v \in V(\cC)]$
be its polynomial ring, where $K$ is a field. A proper interval $[a,b]$ is called an \textit{inner interval} of $\cC$ if all cells contained in $[a,b]$ belong to $\cC$.  

If $[a,b]$ is an inner interval of $\cC$ with diagonal corners $a$ and $b$, and anti-diagonal corners $c$ and $d$, then we consider the binomial $x_a x_b - x_c x_d \in S_\cC$, which is called an \textit{inner 2-minor} of $\cC$. In particular, if $[a,b]$ is a cell of $\cC$, then $x_a x_b - x_c x_d$ is called an \textit{adjacent 2-minor} of $\cC$. 

The ideal $I_{\mathrm{adj}}(\cC)$ in $S_\cC$, generated by all adjacent $2$-minors of $\cC$, is called the \textit{adjacent 2-minor ideal} of $\cC$. We denote by $\cG(I_{\mathrm{adj}}(\cC))$ the set of all binomial generators of $I_{\mathrm{adj}}(\cC)$ of the form $x_a x_b - x_c x_d$ where $[a,b]$ is a cell.  

Similarly, the ideal $I_{\cC}$ in $S_\cC$, generated by all inner $2$-minors of $\cC$, is called the \textit{inner 2-minor ideal} of $\cC$. Clearly, we have $I_{\mathrm{adj}}(\cC) \subseteq I_{\cC}$.  

\begin{Example}\rm
Consider the polyomino $\cP$ in Figure~\ref{Figure: Example weakly conn. coll. of cells}~(B), and assume that 
$V(\cP) = [(1,1),(3,3)]$. Then the adjacent $2$-minor ideal of $\cP$ is
\begin{align*}
I_{\mathrm{adj}}(\cP) = \big(&x_{2,2}x_{1,1} - x_{2,1}x_{1,2}, \;
x_{3,2}x_{2,1} - x_{3,1}x_{2,2}, \;
x_{2,3}x_{1,2} - x_{2,2}x_{1,3}\big).
\end{align*}

\noindent The ideal $I_{\cP}$ of the inner $2$-minors of $\cP$ is generated by
\begin{align*}
x_{2,2}x_{1,1} - x_{2,1}x_{1,2}, \;
x_{3,2}x_{1,1} - x_{3,1}x_{1,2}, \;
x_{3,2}x_{2,1} - x_{3,1}x_{2,2}, \;
x_{2,3}x_{1,1} - x_{2,1}x_{1,3}, \;
x_{2,3}x_{1,2} - x_{2,2}x_{1,3}.
\end{align*}
\end{Example}

Finally, let us recall an important property. If $\cC$ is a collection of cells that can be written as a union of collections of cells $\cC_1, \dots, \cC_r$ such that $V(\cC)$ is the disjoint union of the vertex sets $V(\cC_j)$, with $j \in [r]$, then the following isomorphisms hold:
\[
S_{\cC}/I_{\mathrm{adj}}(\cC) \cong \bigotimes_{i=1}^r \left(S_{\cC_i}/I_{\mathrm{adj}}(\cC_i)\right),
\qquad 
S_{\cC}/I_{\cC} \cong \bigotimes_{i=1}^r \left(S_{\cC_i}/I_{\cC_i}\right).
\]
Most of the commonly studied algebraic properties (such as Cohen–Macaulayness and Gorensteinness) are preserved under tensor products. Hence, it is not restrictive to consider only weakly connected collections of cells. Throughout this paper, whenever we refer to a ``collection of cells'', we mean a weakly connected collection of cells.

  \section{Minimal primes of adjacent $2$-minor ideals}\label{Sec: minimal primes}

This section is devoted to studying the minimal prime ideals of the adjacent $2$-minor ideals of a collection of cells. In what follows, we generalize \cite[Theorem~3.2]{HH} to a more general setting. Before doing so, we introduce some preliminary algebraic and combinatorial tools. We begin by recalling the definition of the lattice ideal of a collection of cells, as given in \cite{Q}.

Let $\cC$ be a collection of cells. For each $a \in V(\cC)$, denote by $\mathbf{v}_a$ the vector in $\mathbb{Z}^{|V(\cC)|}$ having $1$ in the coordinate indexed by $a$ and $0$ in all other coordinates. If $C=[a,b] \in \cC$ is a cell with $a,b$ as its diagonal corners and $c,d$ as its anti-diagonal corners, we set $\mathbf{v}_{[a,b]} = \mathbf{v}_a + \mathbf{v}_b - \mathbf{v}_c - \mathbf{v}_d \in \mathbb{Z}^{|V(\cC)|}$. We define $\Lambda_\cC$ as the sublattice of $\mathbb{Z}^{|V(\cC)|}$ generated by the vectors $\mathbf{v}_C$ for all $C \in \cC$.  Observe that the rank of $\Lambda_\cC$ is equal to $\vert \cC\vert$.\\
Let $n = |V(\cC)|$. For $\mathbf{v} \in \mathbb{N}^n$, we denote by $x^{\mathbf{v}}$ the monomial in $S_\cC$ having $\mathbf{v}$ as its exponent vector. For $\mathbf{e} \in \mathbb{Z}^n$, we denote by $\mathbf{e}^+$ the vector obtained from $\mathbf{e}$ by replacing all negative components with zero, and by $\mathbf{e}^- = -(\mathbf{e} - \mathbf{e}^+)$ the corresponding non-positive part.  \\
Let $L_\cC$ be the lattice ideal of $\Lambda_\cC$, that is, the following binomial ideal in $S_\cC$:
\[
L_\cC = (\{x^{\mathbf{e}^+} - x^{\mathbf{e}^-} \mid \mathbf{e} \in \Lambda_\cC\}).
\]

Given a collection of cells $\cC$, the ideal $L_\cC$ is a prime ideal, see \cite[pp. 288]{Q}. Moreover, it is shown in \cite[Theorem~3.6]{Q} that $I_\cC$ is a prime ideal if and only if $I_\cC = L_\cC$. 

\begin{Proposition}\label{Prop: LC is a minial prime}
Let $\cC$ be a collection of cells. Then $L_{\cC}$ is a minimal prime of $I_{\mathrm{adj}}(\cC)$. Moreover, $\mathrm{ht}(I_{\mathrm{adj}}(\cC)) \leq |\cC|$.
\end{Proposition}

\begin{proof}
Suppose that there exists a prime ideal $\mathfrak{p}$ such that $I_{\mathrm{adj}}(\cC) \subseteq \mathfrak{p} \subseteq L_{\cC}.$ We show that $L_{\cC} \subseteq \mathfrak{p}$.  
Let $f \in L_{\cC}$. By \cite[Proposition~1.1]{HS}, we have
\[
L_{\cC} = I_{\mathrm{adj}}(\cC) : \left(\prod_{a \in V(\cC)} x_a\right)^{\infty}.
\]
Hence, there exists a monomial $u \in S_{\cC}$ such that $uf \in I_{\mathrm{adj}}(\cC)$.  
Since $I_{\mathrm{adj}}(\cC) \subseteq \mathfrak{p}$, it follows that $uf \in \mathfrak{p}$.  
If $f \notin \mathfrak{p}$, then by the primality of $\mathfrak{p}$ we must have $u \in \mathfrak{p}$.  
Again, by the primality of $\mathfrak{p}$, this implies that at least one variable appearing in $u$ belongs to $\mathfrak{p}$.  
However, since $\mathfrak{p} \subseteq L_{\cC}$, this would mean that $L_{\cC}$ contains a variable, which is a contradiction, as $L_{\cC}$ does not contain any variable.\\
Now, it is well known from \cite[Theorem 2.1 (b)]{ES} that, if $\Lambda$ is a lattice, then its lattice ideal $L_{\Lambda}$ has height equal to the rank of $\Lambda$. In our case, the rank of $\Lambda_\cC$ is equal to $\vert \cC\vert$. Then $\mathrm{ht}(I_{\mathrm{adj}}(\cC)) \leq \mathrm{ht}(L_\cC) = |\cC|$.
\end{proof}

We now recall the notion of an admissible set, introduced in~\cite{HH}, which plays a central role in the description of the minimal primes of adjacent $2$-minor ideals.

\begin{Definition}\rm \label{Defn: admissible set}
Let $\cC$ be a collection of cells. A subset $W \subseteq V(\cC)$ is called an \textit{admissible set} of $\cC$ if, for every cell $C$ of $\cC$, either $W \cap C = \emptyset$ or $W \cap C$ contains at least one edge of $C$.
\end{Definition}

The following example illustrates the definition.

\begin{Example}\rm
Let $\cP$ be the polyomino shown in Figure~\ref{Figure: Example weakly conn. coll. of cells} (B).
Consider $W = \{(1,3),(2,3)\}.$ Then $W$ is an admissible set of $\cP$. Indeed, we have $W \cap A = W \cap B = \emptyset$, and $W \cap C$ is an edge of $C$.
Now consider $W' = \{(1,1),(2,1)\}.$ This set is not admissible, since $W' \cap B = \{(2,1)\}$.
\end{Example}

Given an admissible set, we can associate to it a subcollection of cells and a naturally defined prime ideal, as follows.

\begin{Definition}\rm Let $W$ be an admissible set of $\cC$. We define the following collection of cells
    $$
        \cC_W=\{C \in \cC: C\cap W=\emptyset\},
    $$
    and the associated ideal
    $$
        P_{W}(\cC)= (x_w:w\in W)+ L_{\cC_W},
    $$
    where $L_{\cC_W}$ is the lattice ideal attached to $\cC_W$.
\end{Definition}

The next remark records two basic properties of the ideal $P_W(\mathcal{C})$ that will be used repeatedly in what follows.

\begin{Remark}\label{rem:primePW}\rm
For any admissible set $W \subseteq V(\mathcal{C})$, the ideal
$P_W(\mathcal{C})$ is prime. This follows from the fact that
$(x_w : w \in W)$ and the lattice ideal $L_{\mathcal{C}_W}$ are prime ideals involving
disjoint sets of variables. Moreover, by construction,
$P_W(\mathcal{C})$ contains $I_{\mathrm{adj}}(\mathcal{C})$.
Finally, we have
\[
\mathrm{ht}\big(P_W(\mathcal{C})\big)
=
|W| + |\mathcal{C}_W|,
\]
since $\mathrm{ht}\big(L_{\mathcal{C}_W}\big)=|\mathcal{C}_W|$, as discussed
in the last part of the proof of Proposition~\ref{Prop: LC is a minial prime}.
\end{Remark}

The following theorem shows that all minimal primes of
$I_{\mathrm{adj}}(\mathcal{C})$ arise from admissible sets.

     \begin{Theorem}\label{Thm: admissible set}
         Let $\cC$ be a collection of cells and $\mathfrak{p}$ be a minimal prime of $I_{\mathrm{adj}}(\cC)$. Then there exists an admissible set $W$ of $\cC$ such that $\mathfrak{p}=P_{W}(\cC)$.
     \end{Theorem}
      
    \begin{proof}
   By Proposition \ref{Prop: LC is a minial prime}, $L_{\cC}$ is a minimal prime ideal of $I_{\mathrm{adj}}(\cC)$. If $\mathfrak{p} = L_{\cC}$, then $\mathfrak{p} = P_{W}(\cC)$ by setting $W=\emptyset$. We now assume that $\mathfrak{p} \neq L_{\cC}$. Let 
\[
W = \{a \in V(\cC) : x_a \in \mathfrak{p}\}.
\]
We first show that $W \neq \emptyset$. Suppose, for the sake of contradiction, that $W = \emptyset$, i.e., that $\mathfrak{p}$ contains no variables. Let $f \in L_{\cC}$. By \cite[Proposition 1.1]{HS}, there exists a monomial $u$ in $S_{\cC}$ such that $u f \in I_{\mathrm{adj}}(\cC)$, and thus $u f \in \mathfrak{p}$. If $f \notin \mathfrak{p}$, then $u \in \mathfrak{p}$ by the primality of $\mathfrak{p}$, implying that $\mathfrak{p}$ contains a variable, contradicting the assumption that $W = \emptyset$. Therefore, $f \in \mathfrak{p}$, i.e., $L_{\cC} \subseteq \mathfrak{p}$. By Proposition \ref{Prop: LC is a minial prime} and the minimality of $\mathfrak{p}$, we must have $\mathfrak{p} = L_{\cC}$, contradicting the assumption $\mathfrak{p} \neq L_{\cC}$. Hence, $W \neq \emptyset$.\\
We now show that $W$ is an admissible set of $\mathcal{C}$.  Let $C$ be a cell in $\mathcal{C}$ with $V(C) = \{a,b,c,d\}$ such that $a, b$ and $c, d$ are the diagonal and anti-diagonal corners, respectively, and $W \cap C \neq \emptyset$. Without loss of generality, assume that $a \in W \cap C$. Note that $x_ax_b - x_cx_d \in I_{\mathrm{adj}}(\mathcal{C}) \subseteq \fp$, and since $a \in W$, we have $x_a \in \fp$, which implies that $x_cx_d \in \fp$. Since $\fp$ is prime, either $x_c \in \fp$ or $x_d \in \fp$, that is, either $c \in W$ or $d \in W$. Thus, either $\{a,c\} \subseteq W \cap C$ or $\{a,d\} \subseteq W \cap C$. This implies that $W \cap C$ contains an edge, and hence, $W$ is an admissible set of $\mathcal{C}$.

We now prove that $\mathfrak{p} = P_W(\cC)$. Since $I_{\mathrm{adj}}(\cC_W) \subseteq L_{\cC_W}$, we have
\[
I_{\mathrm{adj}}(\cC) \subseteq (x_w : w \in W) + L_{\cC_W} = P_W(\cC).
\]
Since $P_W(\cC)$ is prime, see Remark~\ref{rem:primePW}, to conclude that $\mathfrak{p} = P_W(\cC)$, it suffices to show that $P_W(\cC) \subseteq \mathfrak{p}$, and in particular that $L_{\cC_W} \subseteq \mathfrak{p}$. Let $f \in L_{\cC_W}$. By \cite[Proposition 1.1]{HS}, we have
 \[
 L_{\cC_W} = I_{\mathrm{adj}}(\cC_W) : \left(\prod_{a \in V(\cC) \setminus W} x_a\right)^{\infty}.
 \]
So, there exists a monomial $u$ in $K[x_a : a \in V(\cC) \setminus W]$ such that $uf \in I_{\mathrm{adj}}(\cC)$. Since $I_{\mathrm{adj}}(\cC) \subseteq \mathfrak{p}$, it follows that $uf \in \mathfrak{p}$. Because $\mathfrak{p}$ is prime and $u \notin \mathfrak{p}$, we deduce that $f \in \mathfrak{p}$. In conclusion, $P_W(\cC)$ is a prime ideal satisfying $I_{\mathrm{adj}}(\cC) \subseteq P_W(\cC) \subseteq \mathfrak{p}$. By the minimality of $\mathfrak{p}$, we obtain $P_W(\cC) = \mathfrak{p}$, as desired.
\end{proof}

The following remark shows that the converse of Theorem~\ref{Thm: admissible set} does not hold; namely, not every admissible set defines a minimal prime of the associated adjacent $2$-minor ideal.

   \begin{Remark}\label{ex:notmin}\rm
Consider the polyomino $\cP$ shown in Figure~\ref{Figure: Example weakly conn. coll. of cells} (B). The set $W = \{(1,3),(2,3)\}$ is an admissible set of $\cP$. We briefly show that $P_W(\cP)$ is not a minimal prime of $I_{\mathrm{adj}}(\cP)$. Indeed, by \cite[Theorem~3.3]{CNU}, the ideal $I_{\cP\setminus\{C\}}$ is prime, equivalently, by \cite[Theorem~3.6]{Q}, we have $I_{\cP\setminus\{C\}} = L_{\cP\setminus\{C\}}$. On the other hand, $P_W(\cP) = (x_{1,3},x_{2,3}) + I_{\cP\setminus\{C\}}$.
Since $I_{\mathrm{adj}}(\cP) \subset I_\cP \subset P_W(\cP)$ and $I_\cP$ is a prime ideal by \cite[Theorem~3.3]{CNU}
(more precisely, by Proposition~\ref{Prop: LC is a minial prime}, $I_\cP$ is a minimal prime of $I_{\mathrm{adj}}(\cP)$),
it follows that $P_W(\cP)$ is not a minimal prime of $I_{\mathrm{adj}}(\cP)$. 

However, by using \textit{Macaulay2} (\cite{M2}), in particular the functions 
\href{https://www.macaulay2.com/doc/Macaulay2/share/doc/Macaulay2/PolyominoIdeals/html/_adjacent2__Minor__Ideal.html}{\texttt{adjacent2MinorIdeal}} in \textit{PolyominoIdeals} (\cite{CNJ})
and 
\href{https://macaulay2.com/doc/Macaulay2-1.24.05/share/doc/Macaulay2/Binomials/html/_binomial__Minimal__Primes.html}{\texttt{binomialMinimalPrimes}} in \textit{Binomials} (\cite{Binomials}), we can compute all minimal prime ideals of $I_{\mathrm{adj}}(\cP)$. These are listed below, and one observes that $P_W(\cP)$ does not occur among them.
\[
\begin{aligned}
\fp_1 &= (
x_{2,2}x_{1,1}-x_{2,1}x_{1,2},\,
x_{2,3}x_{1,1}-x_{2,1}x_{1,3},\,
x_{2,3}x_{1,2}-x_{2,2}x_{1,3}, \, x_{3,2}x_{1,1}-x_{3,1}x_{1,2},\\
&x_{3,2}x_{2,1}-x_{3,1}x_{2,2}
) =  I_{\cP}=P_{\emptyset}(\cP),\\
\fp_2 &= (x_{1,2},\,x_{2,1},\,x_{2,2})=P_{W_1}(\cP), \text{ where } W_1=\{(1,2),(2,1),(2,2)\},\\
\fp_3 &= (x_{2,1},\,x_{2,2},\,x_{2,3}) =P_{W_2}(\cP), \text{ where } W_2=\{(2,1),(2,2),(2,3)\},\\
\fp_4 &= (x_{1,2},\,x_{2,2},\,x_{3,2}) =P_{W_3}(\cP), \text{ where } W_3=\{(1,2),(2,2),(3,2)\}.
\end{aligned}
\]
\end{Remark}

\section{From Unmixedness to Complete Intersection for adjacent $2$-minor ideals}\label{Section: CI = Unmixed}

The description of minimal primes in Theorem~\ref{Thm: admissible set} yields the following equivalence among several well-known algebraic properties of adjacent $2$-minor ideals.

\begin{Theorem} \label{Thm: Unmixed = CI}
Let $\cC$ be a collection of cells. Then the following conditions are equivalent:
\begin{enumerate}
\item $I_{\mathrm{adj}}(\cC)$ is a complete intersection;
\item $S_\cC/I_{\mathrm{adj}}(\cC)$ is Gorenstein;
\item $S_\cC/I_{\mathrm{adj}}(\cC)$ is level;
\item $S_\cC/I_{\mathrm{adj}}(\cC)$ is Cohen–Macaulay;
\item $I_{\mathrm{adj}}(\cC)$ is unmixed.
\end{enumerate}
Moreover, if $I_{\mathrm{adj}}(\cC)$ satisfies above conditions, then $\mathrm{ht}(I_{\mathrm{adj}}(\cC)) =|\cC|$.
\end{Theorem}

\begin{proof}
The implications $(1) \Rightarrow (2) \Rightarrow (3) \Rightarrow (4) \Rightarrow (5)$ always hold for graded ideals. We now prove the converse implication $(5) \Rightarrow (1)$. Suppose that $I_{\mathrm{adj}}(\cC)$ is unmixed, that is, all its minimal primes have the same height. By Proposition \ref{Prop: LC is a minial prime}, we have $\mathrm{ht}(I_{\mathrm{adj}}(\cC)) = |\cC|$. This value coincides with $|\mathcal{G}(I_{\mathrm{adj}}(\mathcal{C}))|$. Therefore, $I_{\mathrm{adj}}(\cC)$ is a complete intersection.

Assume that $I_{\mathrm{adj}}(\cC)$ satisfies any of the given conditions. Then $I_{\mathrm{adj}}(\cC)$ is an unmixed ideal and by Proposition~\ref{Prop: LC is a minial prime}, it follows that $\mathrm{ht}(I_{\mathrm{adj}}(\cC)) =|\cC|$.
\end{proof}

As a consequence of Theorem~\ref{Thm: Unmixed = CI}, we derive two necessary combinatorial conditions on the shape of $\mathcal{C}$ that must be satisfied whenever $I_{\mathrm{adj}}(\mathcal{C})$ has one of the equivalent properties listed therein.

We recall that the polyominoes shown on the left and on the right in Figure~\ref{fig: square and X-pento} are known in the literature as the \textit{square tetromino} and the \textit{X-pentomino}, respectively.
\begin{figure}[h]
    \centering
    \includegraphics[scale=0.8]{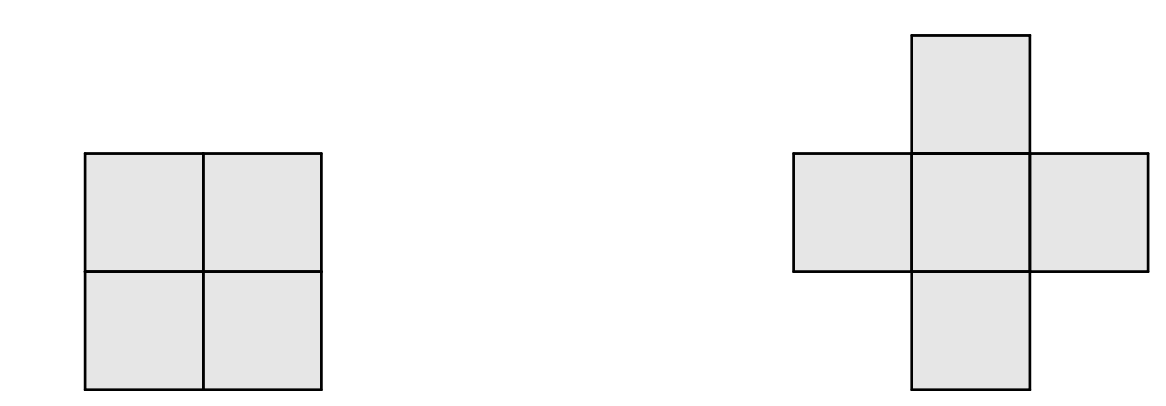}
    \caption{A square tetromino and an $X$-pentomino}
    \label{fig: square and X-pento}
\end{figure}

We next introduce the following definition.
An interval $[a,b]$ with $a=(i,j)$, $b=(k,j)$ and $i<k$ is called a \textit{horizontal edge interval} of $\cC$ if the pairs ${(\ell,j),(\ell+1,j)}$ form edges of cells in $\cC$ for all $\ell\in{i,\dots,k-1}$.
Moreover, if ${(i-1,j),(i,j)}$ and ${(k,j),(k+1,j)}$ do not belong to $E(\cC)$, then $[a,b]$ is called a \textit{maximal horizontal edge interval} of $\cC$. A vertical edge interval and a maximal vertical edge interval can be defined analogously.

In the next two lemmas, we provide two necessary conditions for the adjacent $2$-minor ideal of a collection of cells to satisfy the equivalent conditions in Theorem~\ref{Thm: Unmixed = CI}.

\begin{Lemma}\label{Coro: necessary condition for CI}
Let $\cC$ be a collection of cells. If $I_{\mathrm{adj}}(\cC)$ satisfies the equivalent conditions in Theorem \ref{Thm: Unmixed = CI}, then $\cC$ does not contain a square tetromino.
\end{Lemma}

\begin{proof}
We prove that if $\cC$ contains a square tetromino, then $I_{\mathrm{adj}}(\cC)$ is not unmixed. By Proposition \ref{Prop: LC is a minial prime}, we know that $L_{\cC}$ is a minimal prime of $I_{\mathrm{adj}}(\cC)$ with $\mathrm{ht}(L_\cC) = \vert \cC \vert$. We now construct another minimal prime of $I_{\mathrm{adj}}(\cC)$ whose height is strictly smaller than $\vert \cC \vert$, which implies that $I_{\mathrm{adj}}(\cC)$ is not unmixed. We begin by setting up some notation that will be used throughout the proof, with reference to Figure \ref{fig: A square for necessary cond}. 

\begin{figure}[h]
    \centering
    \includegraphics[scale=1]{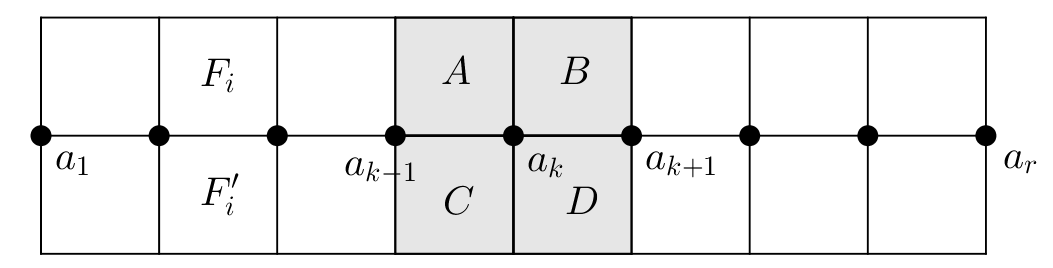}
    \caption{Notation for the proof of Lemma \ref{Coro: necessary condition for CI}}
    \label{fig: A square for necessary cond}
\end{figure}

Let $\cS = \{A,B,C,D\}$ be a square tetromino contained in $\cC$. Let $[a_1,a_r] = \{a_1 < a_2 < \dots < a_r\}$, with $r \geq 3$, be the maximal horizontal edge interval of $\cC$ containing the edges $A \cap C$ and $B \cap D$. Then there exists an index $k \in [r-1]$ such that $a_k$ is the common vertex of $A, B, C$, and $D$. For all $i = 1, \dots, r-1$, we denote by $F_i$ and $F_i'$ the cells having $\{a_i, a_{i+1}\}$ as a common edge. In particular, $F_{k-1} = A$, $F_{k-1}' = C$, $F_{k} = B$, and $F_{k}' = D$.

Set $W = [a_1, a_r]$. Note that $W$ is an addmissible set of $\cC$. With the help of Figure \ref{fig: A square for necessary cond}, we observe the following:

\begin{itemize}
    \item The four cells $F_{k-1}, F_{k-1}', F_{k}$, and $F_{k}'$ all belong to $\cC$;
    \item For all $i \in \{1, \dots, k-2, k+1, \dots, r-1\}$, either $F_i$ or $F_i'$ belongs to $\cC$, since $[a_1, a_r]$ is an edge interval of $\cC$. Note that both may belong to $\cC$ at the same time;
    \item For all $i = 1, \dots, r-1$, neither $F_i$ nor $F_i'$ belongs to $\cC_W$, by definition of $\cC_W$.
\end{itemize}

Consider $P_W(\cC) = L_{\cC_W} + (x_v : v \in W)$. We show that $P_W(\cC)$ is a minimal prime of $I_{\mathrm{adj}}(\cC)$ with $\mathrm{ht}(P_W(\cC)) < \vert \cC \vert$.  First we show that $\mathrm{ht}(P_W(\cC)) < \vert \cC \vert$. Note that $\mathrm{ht}(P_W(\cC)) = \vert \cC_W \vert + \vert W \vert$, as discussed in Remark~\ref{rem:primePW}.

Now, we assign each vertex of $W = [a_1, a_r]$ to a cell in $\mathcal{F} =\cC \cap  \{F_i, F_i' : i \in [r-1]\}$. Define the following map $\phi : W \rightarrow \mathcal{F}$ by
\[
\phi(a_i) =
\begin{cases}
    F_i, & \text{if } F_i \in \cC \text{ and } F_i' \notin \cC, \\[4pt]
    F_i', & \text{if } F_i \notin \cC \text{ and } F_i' \in \cC, \\[4pt]
    F_i, & \text{if } F_i \in \cC \text{ and } F_i' \in \cC,
\end{cases}
\qquad \text{for } i \in [r-1], \quad \text{and let } \phi(a_r) = D = F_{k}'.
\]

The map $\phi$ is clearly injective but not surjective, because there is no vertex $v$ in $W$ such that $\phi(v) = C$. Hence $\vert W \vert < \vert \cF \vert$. Moreover, since $\cC_W = \{C \in \cC : C \cap W = \emptyset\}$ and $\cF = \{C \in \cC : C \cap W \neq \emptyset\}$, we have $\vert \cC_W \vert + \vert \cF \vert = \vert \cC \vert$. Therefore,
\[
\vert \cC_W \vert + \vert W \vert < \vert \cC_W \vert + \vert \cF \vert = \vert \cC \vert,
\]
and thus $\mathrm{ht}(P_W(\cC)) < \vert \cC \vert$, as desired.

To conclude the proof, we show that $P_W(\cC)$ is a minimal prime of $I_{\mathrm{adj}}(\cC)$. It follows from Remark~\ref{rem:primePW} that $P_W(\cC)$ is prime. Suppose that $P_W(\cC)$ is not a minimal prime of $I_{\mathrm{adj}}(\cC)$, so there exists a prime ideal $\mathfrak{p}'$ such that $I_{\mathrm{adj}}(\cC) \subseteq \mathfrak{p}' \subset P_W(\cC)$. It is not restrictive to assume that $\mathfrak{p}'$ is a minimal prime of $I_{\mathrm{adj}}(\cC)$. Then there exists an admissible set $W'$ of $\cC$ such that $\mathfrak{p}' = P_{W'}(\cC) = L_{\cC_{W'}} + (x_v : v \in W') \subset L_{C_W} + (x_v : v \in W)$ by Theorem~\ref{Thm: admissible set}. Hence the variables contained in $\mathfrak{p}'$ form a proper subset of $\{x_v : v \in W\}$, that is, $W' \subset W$. Since $W'$ is an admissible set, the only possibility is that $W' = \emptyset$. Then $\mathfrak{p}' = L_{\cC}$, but this leads to a contradiction because $\mathfrak{p}' \subset P_W(\cC)$ implies $\mathrm{ht}(\mathfrak{p}') < \mathrm{ht}(P_W(\cC)) < \vert \cC \vert$, while $\mathfrak{p}' = L_{\cC}$ gives $\mathrm{ht}(\mathfrak{p}') = \vert \cC \vert$.

This shows that  $P_W(\cC)$ is a minimal prime of $I_{\mathrm{adj}}(\cC)$ with $\mathrm{ht}(P_W(\cC)) < \vert \cC \vert$, whereas the minimal prime $L_{\cC}$ has height equal to $\vert \cC \vert$. This concludes the proof.
\end{proof}

\begin{Lemma}\label{Coro: necessary condition for CI with X pento}
Let $\cC$ be a collection of cells. If $I_{\mathrm{adj}}(\cC)$ satisfies the equivalent conditions in Theorem \ref{Thm: Unmixed = CI}, then $\cC$ does not contain an $X$-pentomino.
\end{Lemma}

\begin{proof}
We show that if $\cC$ contains an $X$-pentomino, then $I_{\mathrm{adj}}(\cC)$ is not unmixed. Let $\cP$ be an $X$-pentomino contained in $\cC$, and set the notation as shown in Figure \ref{fig: X pento notations}.

\begin{figure}[h]
    \centering
    \includegraphics[scale=0.7]{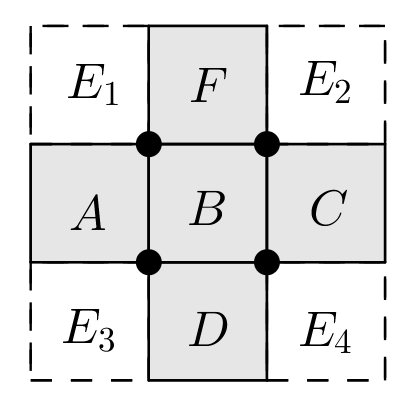}
    \caption{Notation of the $X$-pentomino used in the proof of Lemma \ref{Coro: necessary condition for CI with X pento}.}
    \label{fig: X pento notations}
\end{figure}

Observe that, if $E_i \in \cC$ for some $i \in [4]$, then $\cC$ contains a square tetromino. By Lemma \ref{Coro: necessary condition for CI}, it follows that $I_{\mathrm{adj}}(\cC)$ is not unmixed. Hence, we may assume that $E_i \notin \cC$ for all $i \in [4]$. In this case, $V(B)$ is an admissible set of $\cC$. 

Set $W = V(B)$ and consider $P_W(\cC)$. Note that $\mathrm{ht}(P_W(\cC)) = \vert \cC_W \vert + \vert W \vert$. On the other hand, we have $\vert \cC_W \vert +5= \vert C \vert $ and $\vert W\vert =4$. Therefore, $\mathrm{ht}(P_W(\cC)) =  \vert \cC \vert -1 <  \vert \cC \vert $. Following same arguments as in Lemma \ref{Coro: necessary condition for CI}, one verifies that $P_W(\cC)$ is a minimal prime of $I_{\mathrm{adj}}(\cC)$. Therefore, $I_{\mathrm{adj}}(\cC)$ is not unmixed, as desired. Consequently, $I_{\mathrm{adj}}(\cC)$ does not satisfy any of the equivalent properties listed in Theorem \ref{Thm: Unmixed = CI}. 
\end{proof}

We now show that, for convex collections of cells, the absence of both a square tetromino and an $X$-pentomino provides a characterization of when the adjacent $2$-minor ideal satisfies one of the equivalent properties in
Theorem~\ref{Thm: Unmixed = CI}.

We begin by recalling the definition of convexity and some related notions. A collection of cells $\cC$ is called \emph{row convex} (resp.\ \emph{column convex}) if, for any two cells $A$ and $B$ of $\cC$ lying in the same row (resp.\ the same column), the cell interval $[A,B]$ is contained in $\cC$. If $\cC$ is both row and column convex, then $\cC$ is called \emph{convex} (see Figure~\ref{fig:Ferrer diagram and stack} (A)). 

\noindent Among convex polyominoes, there is a well-known subclass that will be particularly useful for our purpose.  
Let $\cC$ be a convex polyomino, and let $[a,b]$ be the minimal bounding box of $\cC$ (with $c,d$ being the anti-diagonal corners). We say that $\cC$ is a \emph{parallelogram} if $V(\cC)$ contains either the pair $a,b$ or the pair $c,d$ (see Figure~\ref{fig:Ferrer diagram and stack} (B)).

\begin{figure}[h]
    \centering
    \subfloat[]{\includegraphics[scale=0.6]{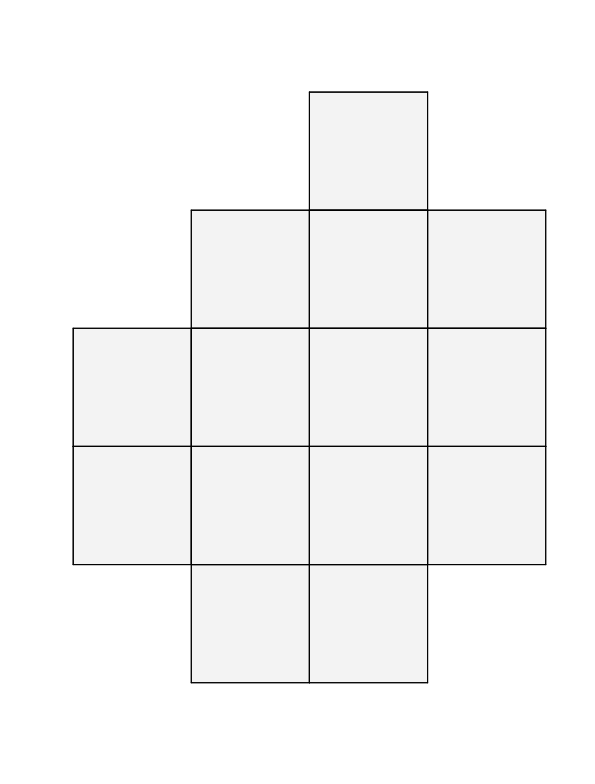}}\qquad\qquad
    \subfloat[]{\includegraphics[scale=0.6]{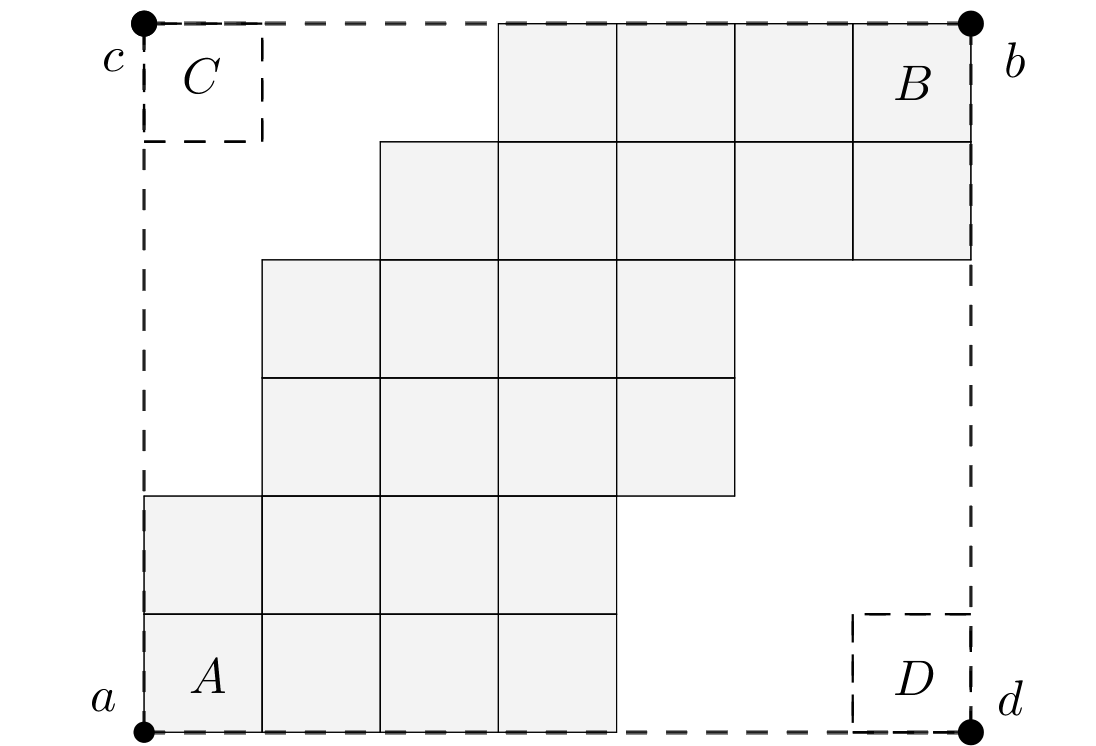}}
    \caption{A convex polyomino and a parallelogram.}
    \label{fig:Ferrer diagram and stack}
\end{figure}

We now describe the possible shapes of convex polyominoes containing neither a square tetromino nor an $X$-pentomino. For the discussion that follows, it is convenient to recall that a polyomino $\cP$ is called a \textit{path polyomino} if either $\cP=\emptyset$, or $\cP$ consists of a cell, or there exist two distinct cells $Q$ and $R$ in $\cP$ such that $\cP$ coincides with the path connecting $Q$ and $R$. Moreover, a \textit{parallelogram path polyomino} is a path polyomino whose shape is that of a parallelogram; see Figure~\ref{Figure: convex polyominoes} (A) for an illustration.

\begin{Discussion}\rm\label{Discussion: convex polyomino}
Let $\cC$ be a convex polyomino not containing a square tetromino or an $X$-pentomino, and $[a,b]$ be the minimal bounding box of $\cC$ (with $c,d$ being the anti-diagonal corners). First, observe that in order to avoid a square tetromino, every maximal interval of $\cC$ must have the form $\cR=[C_1,C_2]$, where $C_1$ and $C_2$ are either in horizontal or vertical position. 

\smallskip
\noindent
\textup{(1)}  If both $a,b\in V(\cC)$ (or $c,d\in V(\cC)$), then $\cC$ has a parallelogram shape; in particular, $\cC$ is a parallelogram path polyomino (see Figure~\ref{Figure: convex polyominoes}~(A)). 

\smallskip
\noindent
\textup{(2)} Assume that $a\notin V(\cC)$ and $b\in V(\cC)$. Since $[a,b]$ is the minimal bounding box of $\cC$, there exist two cells $W$ and $S$ in $\cC$ such that $E(W)\cap [a,c]\neq \emptyset$ and $E(S)\cap [a,d]\neq \emptyset$. The maximal vertical and horizontal intervals of $\cC$ containing $W$ and $S$ intersect in a cell, denoted by $A$; the intersection cannot be empty due to the convexity of $\cC$. 
Set $\cT_W=[W,A]$ and $\cT_S=[S,A]$, which we call the
\emph{West-tail} and the \emph{South-tail}, respectively. 

If $b \in A$, then $\mathcal{C}$ is a parallelogram path consisting of two maximal intervals. Assume now that $b \notin A$. The cells adjacent to $A$ to the North and to the East cannot both belong to $\mathcal{C}$, since otherwise $\mathcal{C}$ would contain an $X$-pentomino.
Let $A'$ denote the unique adjacent cell of $A$ in $\mathcal{C}$ lying either to the North or to the East, and let $a'$ be its lower-left corner. Then the interval $[a',b]$ is the minimal bounding box of $\cP:=\cC\setminus (\cT_W\cup \cT_S)$. The polyomino $\cP$ is convex and satisfies $a',b \in V(\mathcal{P})$; hence, it is a parallelogram path. Therefore, in this case, $\cC$ is either a parallelogram path or consists of two intervals, $\cT_W$ and $\cT_S$ together with a parallelogram path polyomino (see Figure~\ref{Figure: convex polyominoes}~(B)).

\smallskip
\noindent
\textup{(3)} The case $a\in V(\cC)$ and $b\notin V(\cC)$ can be treated analogously. There exist two cells $E$ and $N$ in $\cC$ such that $E$ has an edge in $[d,b]$ and $N$ has an edge in $[c,b]$. The maximal vertical and horizontal intervals of $\cC$ containing $E$ and $N$ intersect in a cell, denoted by $B$. If $a\in B$, then $\cC$ is a parallelogram path consisting of two maximal intervals. Otherwise, denote by $\cT_N=[B,N]$ and $\cT_E=[B,E]$ the \textit{North-tail} and \textit{East-tail}, respectively. Then $\cP:=\cC\setminus (\cT_N\cup \cT_E)$ is a parallelogram path. Hence, in this case, $\cC$ is either a parallelogram path or consists of the two intervals $\cT_N$ and $\cT_E$ together with a parallelogram path polyomino.

\smallskip
\noindent
\textup{(4)} Finally, assume $a,b\notin V(\cC)$. As discussed above, there exist the tails $\cT_S$, $\cT_N$, $\cT_W$, and $\cT_E$, with $\cT_S\cap \cT_W=\{A\}$ and $\cT_N\cap \cT_E=\{B\}$. Since $\cC$ does not contain $X$-pentominoes, the cells $A$ and $B$ are distinct. 

If $E(A) \cap E(B) \neq \emptyset$, then $\mathcal{C}$ consists only of the four tails $\mathcal{T}_S$, $\mathcal{T}_N$, $\mathcal{T}_W$, and $\mathcal{T}_E$. In this case, the remaining parallelogram path is empty; by convention, the empty set is regarded as a collection of cells.

If $E(A)\cap E(B)=\emptyset$, then $\cC$ consists of the four tails $\cT_S$, $\cT_N$, $\cT_W$, and $\cT_E$ together with $\cP:=\cC\setminus (\cT_S\cup\cT_N\cup \cT_W\cup \cT_E)$, where $\cP$ is a parallelogram path (see Figure~\ref{Figure: convex polyominoes}~(C)).

For a convex collection of cells $\mathcal{C}$ containing neither square tetrominoes nor $X$-pentominoes, a similar argument applies. In particular, suppose that $\mathcal{C}$ consists of $n \geq 2$ connected components $\mathcal{C}_1, \dots, \mathcal{C}_n$, arranged from the lower-left to
the upper-right. Then:
\begin{itemize}
\item $\cC_1$ is a parallelogram path, possibly together with a West-tail and a South-tail;
\item $\cC_i$ is a parallelogram path for all $i = 2, \dots, n-1$;
\item $\cC_n$ is a parallelogram path, possibly together with an East-tail and a North-tail.
\end{itemize}
See Figure~\ref{Figure: convex polyominoes}~(D) for an illustration.
\begin{figure}[h]
    \centering
    \subfloat[]{\includegraphics[scale=0.6]{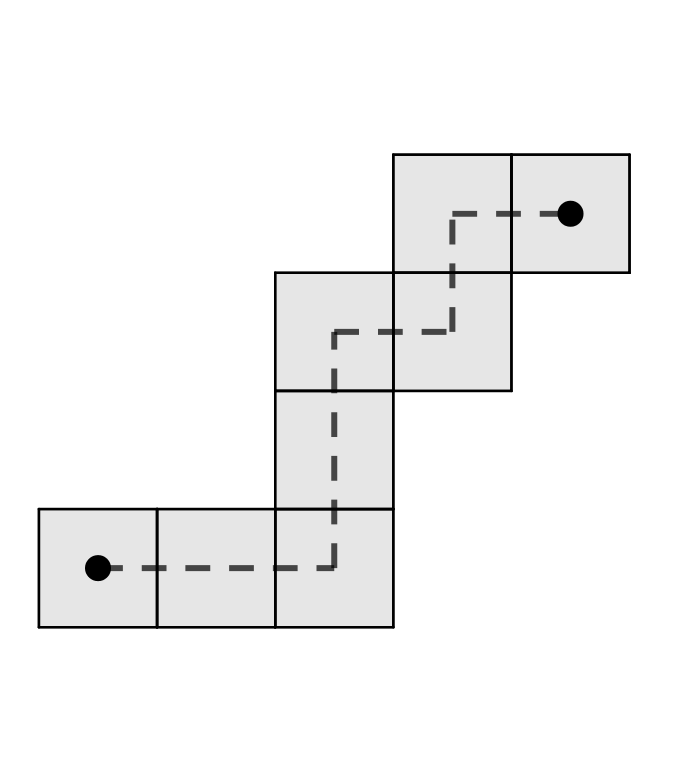}}
    \subfloat[]{\includegraphics[scale=0.6]{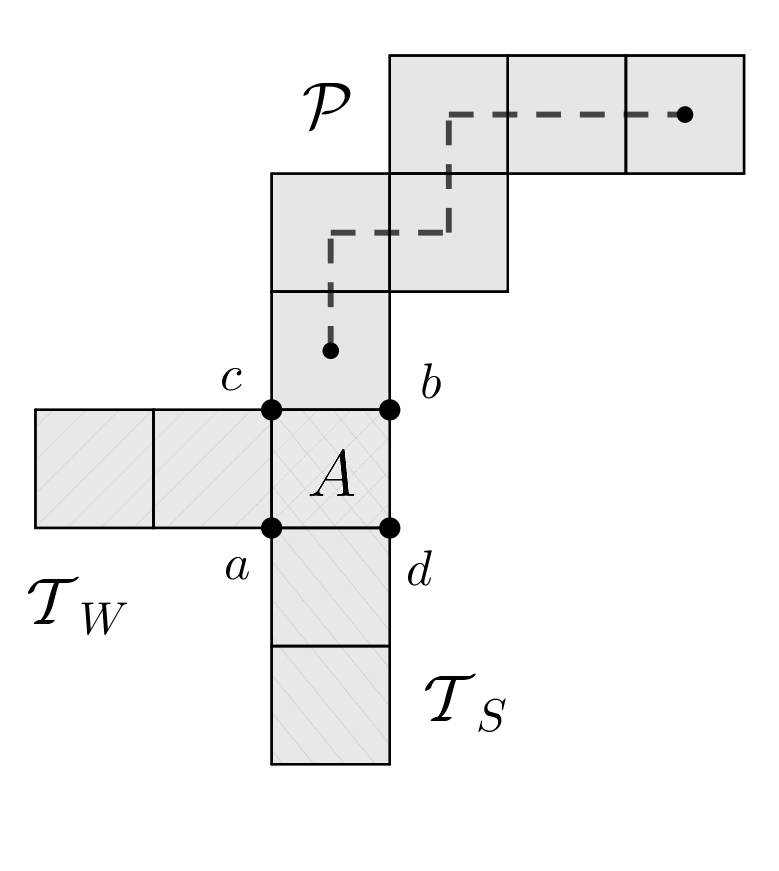}}
    \subfloat[]{\includegraphics[scale=0.6]{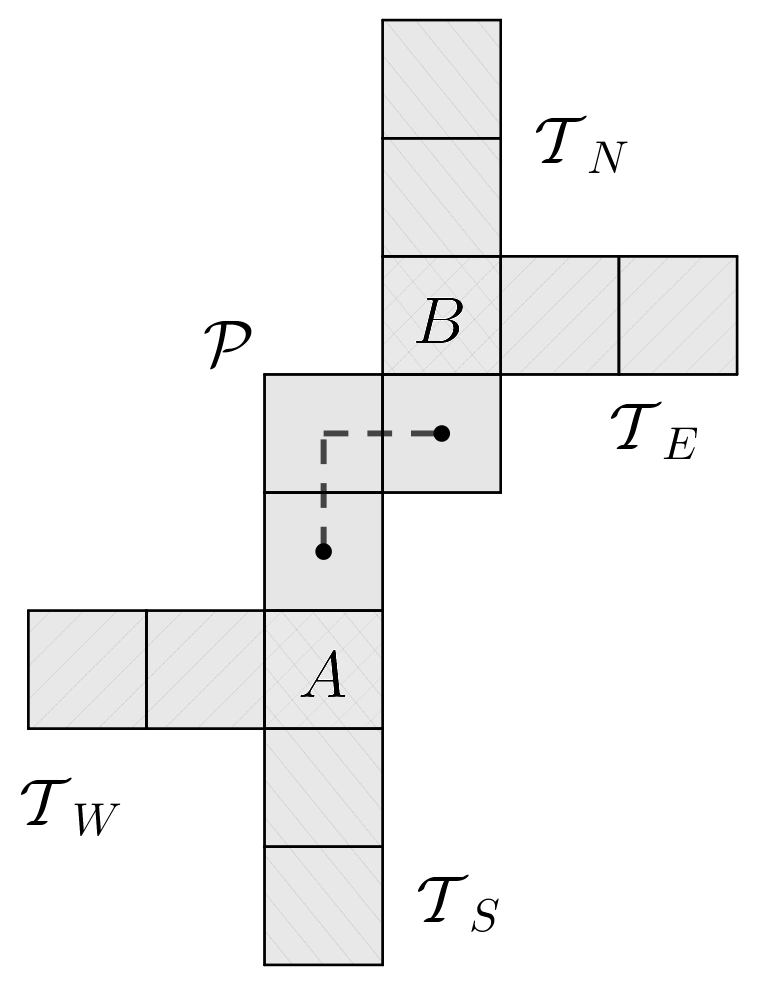}}
    \subfloat[]{\includegraphics[scale=0.5]{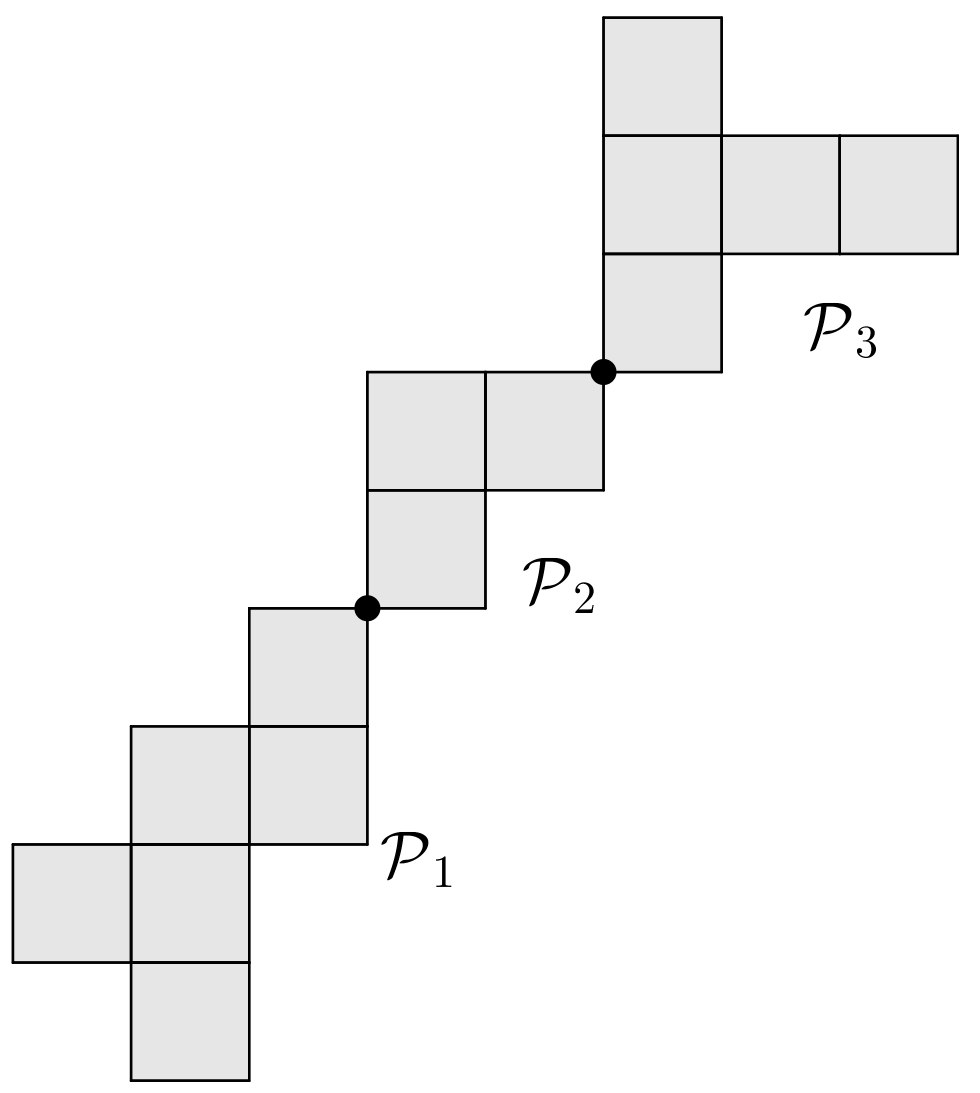}}
    \caption{Possible shapes of a collections of cells containing neither square tetrominoes nor $X$-pentominoes.}
    \label{Figure: convex polyominoes}
\end{figure}
\end{Discussion}

We are now ready to give a combinatorial characterization of all convex collections of cells whose adjacent $2$-minor ideal satisfies any of the equivalent properties listed in
Theorem~\ref{Thm: Unmixed = CI}. Before stating the result, we recall some terminology.
 
 Let $\cC$ be a collection of cells. For any cell $C \in \cC$ with diagonal corners $a$ and $b$, and anti-diagonal corners $c$ and $d$, we refer to $x_a x_b$ as the \emph{diagonal term} and to $x_c x_d$ as the \emph{anti-diagonal term} of $C$. Furthermore, if $<$ denotes a monomial order on $S_{\cC}$, we write $\mathrm{in}_<(f)$ for the initial monomial of a polynomial $f \in S_{\cC}$, and $S(f,g)$ for the $S$-polynomial associated with $f,g \in S_{\cC}$.

\begin{Theorem}\label{Thm: characterization convex CI}
Let $\cC$ be a convex collection of cells. Then $I_{\mathrm{adj}}(\cC)$ satisfies the equivalent conditions in Theorem~\ref{Thm: Unmixed = CI} if and only if $\cC$ contains neither a square tetromino nor an $X$-pentomino.
\end{Theorem}

\begin{proof}
    $\Rightarrow)$ The implication follows from Lemmas~\ref{Coro: necessary condition for CI} and~\ref{Coro: necessary condition for CI with X pento}.

\smallskip
\noindent
$\Leftarrow)$ Assume that $\cC$ contains neither a square tetromino nor an $X$-pentomino. We first suppose that $\cC$ is a polyomino and analyze
two situations, depending on its shape as described in
Discussion~\ref{Discussion: convex polyomino}.

\smallskip
\noindent
\textbf{Situation A.} $\cC$ is a parallelogram path.

Let $[a,b]$ be its minimal bounding box. Up to reflection with respect to the $x$-axis, we may assume that $V(\cC)$ contains both $a$ and $b$, that is, $\cC$ has the shape shown in Figure~\ref{Figure: convex polyominoes}~(A).
Consider the lexicographic order induced by $x_{ij} > x_{kl}$ if $j > l$, or $i < k$ when $j = l$. Then the leading monomial of each generator of $I_{\mathrm{adj}}(\cC)$ is the corresponding anti-diagonal term.

Due to the specific shape of $\cC$, we have $\gcd(\mathrm{in}_<(f), \mathrm{in}_<(g)) = 1$ for all $f,g \in \cG(I_{\mathrm{adj}}(\cC))$. Hence, each $S(f,g)$ reduces to zero modulo $\cG(I_{\mathrm{adj}}(\cC))$, and therefore $\mathrm{in}_<(I_{\mathrm{adj}}(\cC))$is a complete intersection. Consequently, $I_{\mathrm{adj}}(\cC)$ itself is a complete intersection. Consequently, $I_{\mathrm{adj}}(\cC)$ itself is a complete intersection,
and thus satisfies the equivalent conditions in
Theorem~\ref{Thm: Unmixed = CI}.
In particular, in this case $I_{\mathrm{adj}}(\cC)$ is unmixed, a fact that will be used in the following case.

\smallskip
\noindent
\textbf{Situation B.} $\cC$ is not a parallelogram path.

According to Discussion~\ref{Discussion: convex polyomino}, there are three possible cases, which we analyze below.

\medskip
\noindent
\textbf{Case I.} Assume that $\cC$ is the union of a parallelogram polyomino $\cP$ and the tails $\cT_S$ and $\cT_W$ (see, for instance, Figure~\ref{Figure: convex polyominoes}~(B)). Let $A$ be the cell in $\cT_S \cap \cT_W$, and denote by $a$ its lower-left corner. Set $b = a + (1,1)$, $c = a + (0,1)$, and $d = a + (1,0)$. 
  
  We show that $I_{\mathrm{adj}}(\cC)$ is unmixed. By Proposition~\ref{Prop: LC is a minial prime}, the lattice ideal $L_{\cC}$ is a minimal prime of $I_{\mathrm{adj}}(\cC)$ of height $\vert \cC \vert$. Let $\fp$ be a minimal prime of $I_{\mathrm{adj}}(\cC)$ with $\fp\neq L_{\cC}$. By Theorem~\ref{Thm: admissible set}, there exists a non-empty admissible set $U$ of $\cC$ such that $\fp = P_U(\cC)$, and by Remark~\ref{rem:primePW}, we obtain $\mathrm{ht}(\fp) = \vert U \vert + \vert \cC_U \vert$. To show that $I_{\mathrm{adj}}(\cC)$ is unmixed, it suffices to show that $\vert U \vert + \vert \cC_U \vert = \vert \cC \vert$.

\medskip
\noindent
\textit{Case I.1.}
Assume that $U\cap \{a,d\}=\emptyset$. Set $\cP' = \cT_W \cup \cP$ and $\cP'' = \cT_S \setminus \{A\}$, and define
$$U_1 = U \cap V(\cP') \quad \text{and} \quad U_2 = U \cap V(\cP'').$$
Then $U_1$ and $U_2$ are admissible sets of $\cP'$ and $\cP''$, respectively. Since $\cP'$ is a parallelogram path, Situation~\textup{A} implies that $I_{\mathrm{adj}}(\cP')$ is unmixed. We claim that $P_{U_1}(\cP')$ is a minimal prime of $I_{\mathrm{adj}}(\cP')$. Suppose, to the contrary, that it is not. Then there exists a minimal prime ideal
$\mathfrak{p}'$ such that $I_{\mathrm{adj}}(\cP') \subset \fp' \subset P_{U_1}(\cP').$ By Theorem~\ref{Thm: admissible set}, we have $\fp' = P_{U'}(\cP')$ for some admissible set $U'$ of $\cP'$. Then $U' \subset U_1$. Now, $\tilde{U} := U' \cup U_2$ is an admissible set of $\cC$, and we have $I_{\mathrm{adj}}(\cC) \subset P_{\tilde{U}}(\cC) \subset \fp=P_U(\cC),$
which contradicts the minimality of $\mathfrak{p}$.  Therefore, $P_{U_1}(\cP')$ is indeed a minimal prime of $I_{\mathrm{adj}}(\cP')$, and by Theorem~\ref{Thm: Unmixed = CI} and Remark~\ref{rem:primePW} we obtain
\begin{equation}\label{eq1}
    \vert \cP'_{U_1} \vert + \vert U_1 \vert = \vert \cP' \vert.
\end{equation}
Since $\cP''$ is a parallelogram path, same argument applies to $U_2$ and $\cP''$, yielding 
\begin{equation}\label{eq2}
\vert \cP''_{U_2} \vert + \vert U_2 \vert = \vert \cP'' \vert.  
\end{equation}
Adding \eqref{eq1} and \eqref{eq2}, and observing that $\vert \cP'_{U_1} \vert + \vert \cP''_{U_2} \vert=\vert \cC_U \vert$, we obtain  
$\vert \cC_U \vert + \vert U \vert = \vert \cC \vert$,  
as required.  An analogous argument applies if $U\cap \{a,c\}=\emptyset$. 

\medskip
\noindent
\textit{Case I.2.}  Assume that $U\cap \{a,d\}=\{d\}$. If $c\notin U$, we fall into the previous case where $U\cap \{a,c\}=\emptyset$. Assume therefore that $c \in U$; we show that this leads to a contradiction. Since $U$ is an admissible set of $\cC$ and $c,d\in U$ while $a\notin U$, admissibility forces $b\in U$. For the same reason, the vertices $s$ and $w$ in Figure \ref{fig: prrof of characterization convex first} belong to $U$. We distinguish two cases according to whether $A'$ lies to the North or to the East
of $A$.
\begin{itemize}
    \item[\textup{(i)}]
    \textit{$A'$ lies to the North of $A$}
    (see Figure~\ref{fig: prrof of characterization convex first}~(A)).\\
    Set $U' = U \cap \left(V(\cP \cup \cT_W) \setminus \{d\}\right)$.  Then $P_{U'}(\mathcal{P}\cup \mathcal{T}_W)+L_{\mathcal{T}_S\setminus\{A\}}$
    is a prime ideal and
    \[
    I_{\mathrm{adj}}(\mathcal{C})
    \subset
    P_{U'}(\mathcal{P}\cup \mathcal{T}_W)+L_{\mathcal{T}_S\setminus\{A\}}
    \subsetneq
    P_U(\mathcal{C}),
    \]
    which contradicts the minimality of $P_U(\mathcal{C})$.

     \item[\textup{(ii)}]
    \textit{$A'$ lies to the East of $A$}
    (see Figure~\ref{fig: prrof of characterization convex first}~(B)).\\
    Set $U'' = U \cap \bigl(V(\mathcal{P}\cup \mathcal{T}_S)\setminus\{c\}\bigr).$ Then $P_{U''}(\mathcal{P}\cup \mathcal{T}_S)+L_{\mathcal{T}_W\setminus\{A\}}$
    is a prime ideal and
    \[
    I_{\mathrm{adj}}(\mathcal{C})
    \subset
    P_{U''}(\mathcal{P}\cup \mathcal{T}_S)+L_{\mathcal{T}_W\setminus\{A\}}
    \subsetneq
    P_U(\mathcal{C}),
    \]
    again contradicting the minimality of $P_U(\mathcal{C})$.
\end{itemize}

\begin{figure}[h]
   \centering
   \subfloat[]{\includegraphics[scale=0.8]{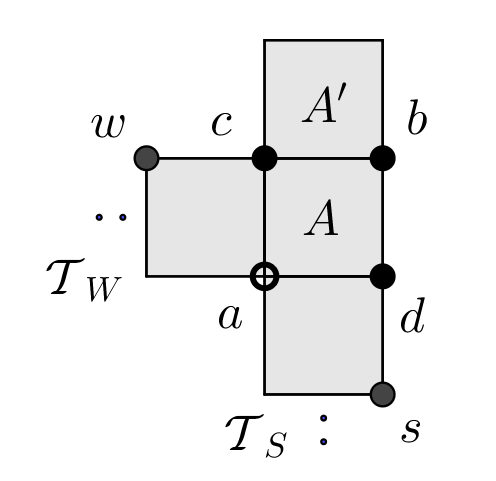}}\qquad
   \subfloat[]{\includegraphics[scale=0.8]{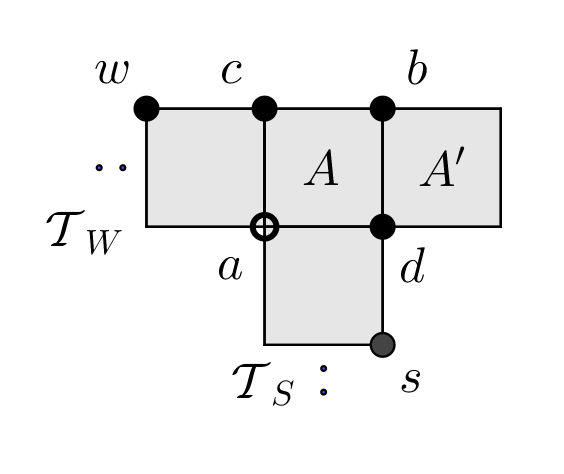}}
   \caption{Intersection of $\cT_S$ and $\cT_W$.}
   \label{fig: prrof of characterization convex first}
\end{figure}

\medskip
\noindent
\textit{Case I.3.} Suppose that $U \cap \{a,d\} = \{a,d\}$. In this case, we distinguish four
subcases to prove the equality  $\vert \cC_U \vert + \vert U \vert = \vert \cC \vert$.  
Let $A_S$ and $A_W$ denote the cells of $\cC$ lying, respectively, to the South and to the West of $A$, and let $A'$ be the cell of $\cC$ that shares either $\{c,b\}$ or $\{b,d\}$ with $A$; see Figure~\ref{fig:prrof of characterization convex}.

\begin{itemize}
\item[\textup{(i)}]
\textit{$c\in U$ and $b\notin U$}
(see Figure~\ref{fig:prrof of characterization convex}).

\noindent Set $\cC(W):=\cT_W \setminus \{A, A_W\}$ and $\cC(S):=\cT_S \setminus \{A, A_S\}$. 
Observe that 
\[
U_W := U \cap V(\cC(W)), \quad 
U_S := U \cap V(\cC(S)) \quad\text{ and }\quad 
U' := U \cap V(\cP)
\]
are admissible sets of $\cC(W)$, $\cC(S)$, and $\cP$, respectively. 
Therefore, the ideals $P_{U_W}(\cC(W))$, $P_{U_S}(\cC(S))$, and $P_{U'}(\cP)$ are well-defined.
Since $\cC(W)$, $\cC(S)$, and $\cP$ are parallelogram paths, the corresponding adjacent $2$-minor ideals are unmixed as discussed in Situation~A. Hence, arguing as in \textit{Case~I.1}, we obtain the following:
\begin{equation}\label{eq3}
    \vert \cC(W)_{U_W} \vert + \vert U_W \vert = \vert \cC(W) \vert,\qquad 
    \vert \cC(S)_{U_S} \vert + \vert U_S \vert = \vert \cC(S) \vert,\qquad 
    \vert \cP_{U'} \vert + \vert U' \vert = \vert \cP \vert.
\end{equation}

\noindent Since $\cC = \cC(W) \sqcup \cC(S) \sqcup  \cP \sqcup \{A, A_W, A_S\}$, we have
 \begin{equation}\label{eq4}
     \vert\cC\vert = \vert \cC(W) \vert+ \vert \cC(S) \vert+\vert \cP \vert +3.
 \end{equation}
Combining \eqref{eq3} and \eqref{eq4}, we obtain
 $$\vert \cC\vert=(\vert \cC(W)_{U_W} \vert+\vert \cC(S)_{U_S} \vert +\vert \cC_{U'} \vert )+ (\vert U_W \vert + \vert U_S \vert + \vert U' \vert+3).$$
 At this point, one can easily observe that $U_W \sqcup  U_S \sqcup U' \sqcup\{a,c,d\}= U$ and 
     \[
 \cC(W)_{U_W} \sqcup \cC(S)_{U_S} \sqcup  \cC_{U'} = (\cC\setminus\{A,A_S,A_W\})_{U_W\sqcup U_S\sqcup U'}=\cC_{U_W\sqcup U_S\sqcup U'\sqcup \{a,c,d\}}=\cC_U.
 \]
 Hence, we conclude that $|\cC|=|\cC_U|+|U|$, which proves the claim.

\begin{figure}[h]
   \centering
   \subfloat{\includegraphics[scale=0.8]{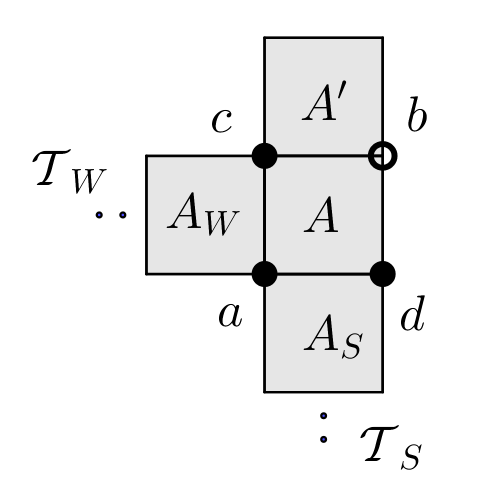}}\qquad
   \subfloat{\includegraphics[scale=0.8]{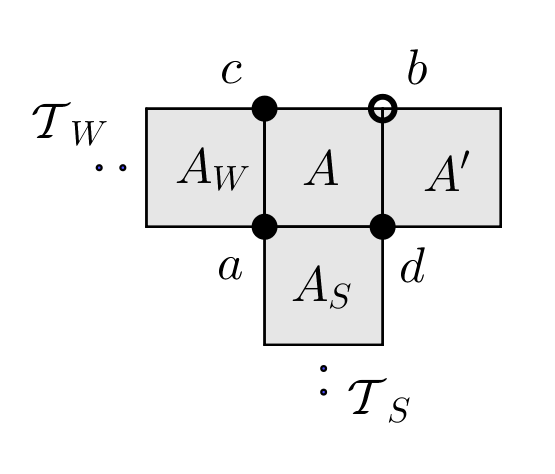}}
   \caption{$c\in U$ and $b\notin U$.}
   \label{fig:prrof of characterization convex}
\end{figure}
\item[\textup{(ii)}]
\textit{$c\notin U$ and $b\in U$.}\\
This case follows from Subcase~\textup{(i)} by interchanging the roles of $b$
and $c$.

\item[\textup{(iii)}]
\textit{$c\in U$ and $b\in U$.}\\
This case is analogous to Subcase~\textup{(i)},  with the roles of $\cP$ and $\{a,b,c\}$ now played by $\cP \setminus \{A'\}$ and $\{a,b,c,d\}$, respectively, and noting that 
$\cC_{\{a,b,c,d\}} = \cC \setminus \{A, A_W, A_S, A'\}$.

\item[\textup{(iv)}]
\textit{$c\notin U$ and $b\notin U$.}\\
In this case, $U \cap \{b,c\} = \emptyset$, and we argue as in \textit{Case~I.1}, and working with $\cT_W\cup \cT_S$ and $\cP\setminus\{A'\}$, which are both parallelogram paths, so the corresponding adjacent $2$-minor ideals are unmixed by Situation~A.
\end{itemize}

\medskip
\noindent
\textit{Case I.4.}
Assume that $U \cap \{a,d\} = \{a\}$.
Since $U$ is admissible, it follows that $c \in U$, hence
$U \cap \{a,c\} = \{a,c\}$.
The conclusion then follows by an argument analogous to \textit{Case~I.3}.

\medskip
\noindent
This completes the proof for the case where $\cC$ is the union of a parallelogram
polyomino $\cP$ and the tails $\cT_S$ and $\cT_W$, and hence
$I_{\mathrm{adj}}(\cC)$ is unmixed.

\medskip
\noindent
\textbf{Case II.}
The case where $\mathcal{C}$ consists of the East-tail and the North-tail
together with a parallelogram polyomino is analogous to Case~I, up to symmetry
(reflection/rotation).

\medskip
\noindent
\textbf{Case III.}
More generally, assume that $\cC$ consists of all four tails together
with a parallelogram polyomino $\cP$
(see, for instance, Figure~\ref{Figure: convex polyominoes}~(C)). Let $A$ and $B$ be the cells in $\cT_S \cap \cT_W$ and $\cT_N \cap \cT_E$, respectively. Denote by $a$ the lower-left corner of $A$, and set 
$b = a + (1,1)$, $c = a + (0,1)$, and $d = a + (1,0)$.  Similarly, let $a'$ be the lower-left corner of $B$, and define $b' = a' + (1,1)$, $c' = a' + (0,1)$, and $d' = a' + (1,0)$.

\begin{itemize}
    \item If $\cP \neq \emptyset$, then we may regard $\cC$ as consisting of the two tails $\cT_W$ and $\cT_S$, together with the polyomino $\cP' := \cP \cup \cT_N \cup \cT_E.$ By Case~I, we know that $I_{\mathrm{adj}}(\cP')$ is unmixed. The claim follows by arguing exactly as in Case~I, with the role of $\cP$ in that case now played by $\cP'$.
    
\item If $\mathcal{P} = \emptyset$, then $A$ and $B$ are adjacent. Up to rotation or reflection of $\mathcal{C}$, we may assume that $B$ lies to the North of $A$, so that $A \cap B$ is the edge $\{b,c\}$.
In this situation, the proof proceeds exactly as in Case~I: one examines whether
$b$ or $c$ belongs to $U$, and then considers the corresponding subcases
determined by the membership of $a,d,c',d'$ in $U$.
To avoid lengthening the proof further, we omit this straightforward verification;
see also the folder \texttt{Examples\_CaseIII} in~\cite{N} for explicit examples.
\end{itemize}

\medskip
\noindent
This completes the proof in the case where \(\cC\) is a polyomino.

To complete the argument, we now consider the case in which $\cC$ is a collection of cells that is not a polyomino. According to the last paragraph of Discussion~\ref{Discussion: convex polyomino}, we may assume that $\cC$ consists of connected components arranged from the lower-left to the upper-right, starting with a polyomino $\cP_1$ formed by a parallelogram path with at most West and South tails, followed by a finite number $n-2$ of parallelogram paths
$\cP_2, \dots, \cP_{n-1}$, and ending with a polyomino $\cP_n$ consisting of a parallelogram path with at most East and North tails (see Figure~\ref{Figure: convex polyominoes}~(D) for an example). 

\medskip
\noindent
\textbf{Case A.} Assume that all connected components are parallelogram paths.

Let $[a_1,b_1]$ and $[a_n,b_n]$ be the minimal bounding boxes of $\cP_1$ and $\cP_n$, respectively, and denote by $c_1,d_1$ and $c_n,d_n$ their anti-diagonal corners.  

\medskip
\noindent
\textit{Case A.1.} If $V(\cP_1)$ and $V(\cP_n)$ contain $a_1,b_1$ and $a_n,b_n$, respectively, then $\mathrm{in}_<(I_{\mathrm{adj}}(\cC))$ is a complete intersection with respect to the same monomial order $<$ introduced in Situation~A. Consequently, $I_{\mathrm{adj}}(\cC)$ itself is a complete intersection, as required. 

\begin{figure}[h]
    \centering
    \subfloat[]{\includegraphics[scale=0.7]{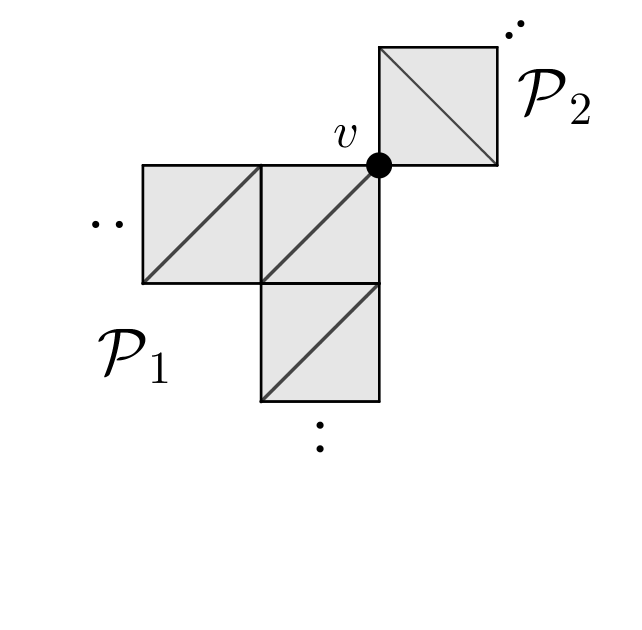}}\qquad
    \subfloat[]{\includegraphics[scale=0.7]{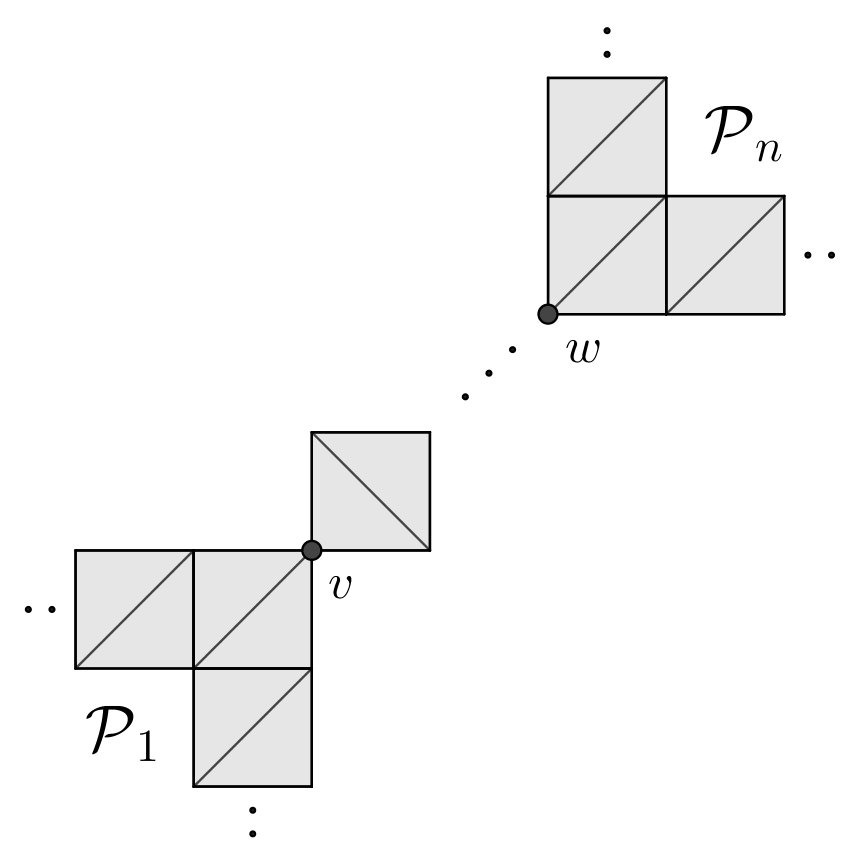}}\qquad
        \subfloat[]{\includegraphics[scale=0.7]{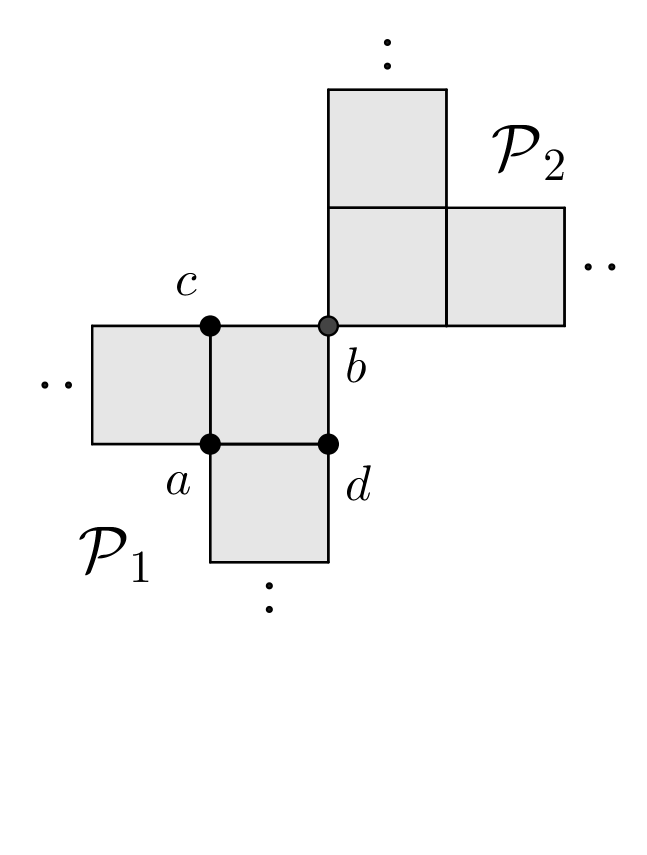}}
    \caption{Possible shapes of $\cC$.}
    \label{Figure: possible orientation P1}
\end{figure}

\medskip
\noindent
\textit{Case A.2.} If $V(\cP_1)$ contains $c_1,d_1$ (see Figure~\ref{Figure: possible orientation P1} (A)) and $V(\cP_n)$ contains $a_n,b_n$, then we set $V(\cP_1)\cap V(\cP_2)=\{v\}$, and define the following total orders on $\ZZ^2$:
\begin{itemize}
    \item $(k,l)<^1(i,j)$ if $k < i$, or $l < j$ when $i = k$.
    \item $(k,l)<^2(i,j)$ if $i < k$, or $l < j$ when $i = k$.
\end{itemize}
Consider the lexicographic order $<$ induced by the following ordering of the variables:
\[
x_p < x_q \iff 
\begin{cases}
    &p \in V(\cP_1),\ q \in V(\cC\setminus \cP_1)\setminus\{v\}, \\
    &p,q \in V(\cP_1) \text{ and } p <^1 q,  \\
    &p,q \in V(\cC\setminus\cP_1)\setminus\{v\} \text{ and } p <^2 q.
\end{cases}
\]
With respect to this order, the initial terms of the generators of $I_{\mathrm{adj}}(\cP_1)$ are diagonal, while those of $I_{\mathrm{adj}}(\cC\setminus \cP_1)$ are anti-diagonal. Consequently, $\mathrm{in}_<(I_{\mathrm{adj}}(\cC))$ is a complete intersection, and hence, $I_{\mathrm{adj}}(\cC)$ itself is a complete intersection, as required.

\medskip
\noindent
\textit{Case A.3.}
If $V(\cP_1)$ contains $a_1,b_1$ and $V(\cP_n)$ contains $c_n,d_n$, then, after reflecting $\cC$ and relabeling the connected components, we reduce to the situation considered in Case~A.2.

\medskip
\noindent
\textit{Case A.4.} If $V(\cP_1)$ contains $c_1,d_1$ and $V(\cP_n)$ contains $c_n,d_n$, then we need to distinguish two sub-cases according to the number of connected components of $\cC$.

\begin{itemize}
    \item If $n>2$ (see Figure~\ref{Figure: possible orientation P1} (B)), set $V(\cP_1)\cap V(\cP_2)=\{v\}$ and $V(\cP_{n-1})\cap V(\cP_n)=\{w\}$, and define $<^1$ and $<^2$ as in Case~A.2. Consider the lexicographic order $<$ induced by the following ordering of the variables, for $p,q\in V(\cC)$:
\[
x_p < x_q \iff 
\begin{cases}
    &p \in V(\cP_1),\ q \in V(\cC\setminus \cP_1)\setminus\{v\}, \\
    &p \in V(\cC\setminus\cP_n),\ q \in V(\cP_n)\setminus\{w\}, \\
    &p,q \in V(\cP_1) \text{ and } p <^1 q,  \\
    &p,q \in V(\cC\setminus(\cP_1\cup \cP_n))\setminus\{v\} \text{ and } p <^2 q,\\
    &p,q \in V(\cP_n)\setminus\{w\} \text{ and } p <^1 q.
\end{cases}
\]
With respect to this order, the initial terms of the generators of $I_{\mathrm{adj}}(\cP_1)$ and $I_{\mathrm{adj}}(\cP_n)$ are diagonal, whereas those of $I_{\mathrm{adj}}(\cC\setminus (\cP_1\cup \cP_n)$ are anti-diagonal. Hence, $\mathrm{in}_<(I_{\mathrm{adj}}(\cC))$ is a complete intersection. Consequently, $I_{\mathrm{adj}}(\cC)$ itself is a complete intersection, as required.

\item
If $n=2$ (see Figure~\ref{Figure: possible orientation P1}~(C)), it is easy to see that $\mathrm{in}_{<}(I_{\mathrm{adj}}(\cC))$ is not a complete intersection for any monomial order $<$. In this case, let $V$ and $H$ denote the vertical and horizontal maximal intervals forming $\cP_1$, and let $V\cap H=\{A\}$. If $\fp$ is a minimal prime of $I_{\mathrm{adj}}(\cC)$, we must show that $\mathrm{ht}(\fp) = \vert \cC \vert$, that is, $\vert U \vert + \vert \cC_U \vert = \vert \cC \vert$ for some admissible set $U$ of $\cC$. This follows by arguing as in Case~I on the vertices of $A$, recalling that,
as shown earlier, the adjacent $2$-minor ideals of collections of cells contained in $H\cup \cP_2$ or $V\cup \cP_2$ are unmixed.
\end{itemize}

\medskip
\noindent
\textbf{Case B.}
Assume that not all connected components of $\cC$ are parallelogram paths.

We distinguish the following subcases, depending on the shapes of the first and last connected components.

\medskip
\noindent
\textit{Case B.1.} 
Suppose that $\cP_1$ consists of a parallelogram path with only West and South tails, while $\cP_n$ is a parallelogram path.
In this case, the argument is identical to that of Case~I, with the role of Situation~(A) in that case (namely, when $\cC$ itself is a parallelogram path) now played by Case~A.

\medskip
\noindent
\textit{Case B.2.} Suppose that $\cP_1$ is a parallelogram path and that $\cP_n$ consists of a parallelogram path with only North and East tails. By applying suitable rotations and reflections to $\cC$, this situation reduces to Case~B.1.

\medskip
\noindent
\textit{Case B.3.} If neither $\cP_1$ nor $\cP_n$ is a parallelogram path, then the argument proceeds exactly as in Case~III.

\medskip
\noindent
This completes the proof.
\end{proof}

\begin{Remark}\rm
In the folder \texttt{Example\_minimal\_prime\_convex} in \cite{N}, the reader can find several examples illustrating the admissible sets associated with the minimal primes of convex polyominoes that contain neither a square tetromino nor an $X$-pentomino.
\end{Remark}

\begin{Remark}\rm\label{Rmk: counter-example convex}
We provide a counterexample in the non-convex case by exhibiting a simple non-convex polyomino that contains neither a square tetromino nor an $X$-pentomino, but whose adjacent $2$-minor ideal is not unmixed. Consider the polyomino $\cP$ shown in Figure~\ref{fig: non-convex}.

\begin{figure}[h]
   \centering
   \includegraphics[scale=1]{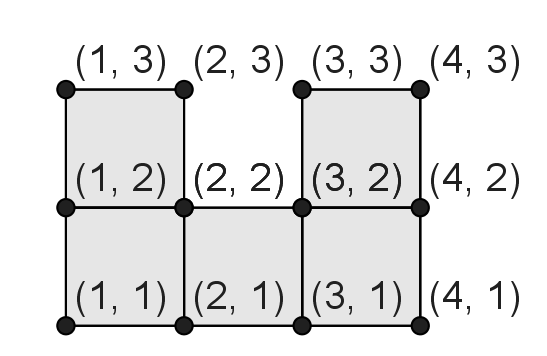}
   \caption{A non-convex polyomino.}
   \label{fig: non-convex}
\end{figure}

\noindent Using \textit{Macaulay2} (\cite{CNJ, M2, Binomials}), we can efficiently compute all minimal prime ideals of the adjacent $2$-minor ideal of $\cP$, which are listed below:
\[
\begin{aligned}
& \fp_1 = I_{\cP} = (x_{2,2}x_{1,1}-x_{2,1}x_{1,2},\,x_{2,3}x_{1,1}-x_{2,1}x_{1,3},\,x_{2,3}x_{1,2}-x_{2,2}x_{1,3},\,x_{3,2}x_{1,1}-x_{3,1}x_{1,2},\,x_{3,2}x_{2,1}-\\
&x_{3,1}x_{2,2},\, x_{4,2}x_{1,1}-x_{4,1}x_{1,2},\,x_{4,2}x_{2,1}-x_{4,1}x_{2,2},\,x_{4,2}x_{3,1}-x_{4,1}x_{3,2},\,x_{4,3}x_{3,1}-x_{4,1}x_{3,3},\,x_{4,3}x_{3,2}-x_{4,2}x_{3,3}),\\
& \fp_2 = (x_{4,2}x_{1,1}-x_{4,1}x_{1,2},\,x_{4,2}x_{2,1}-x_{4,1}x_{2,2},\,x_{4,2}x_{3,1}-x_{4,1}x_{3,2},\,x_{4,3}x_{3,1}-x_{4,1}x_{3,3},\,x_{4,3}x_{3,2}-x_{4,2}x_{3,3}),\\
& \fp_3 = (x_{1,2},\,x_{2,1},\,x_{2,2},\,x_{4,2}x_{3,1}-x_{4,1}x_{3,2},\,x_{4,3}x_{3,1}-x_{4,1}x_{3,3},\,x_{4,3}x_{3,2}-x_{4,2}x_{3,3}),\\
& \fp_4 = (x_{2,1},\,x_{2,2},\,x_{2,3},\,x_{4,2}x_{3,1}-x_{4,1}x_{3,2},\,x_{4,3}x_{3,1}-x_{4,1}x_{3,3},\,x_{4,3}x_{3,2}-x_{4,2}x_{3,3}),\\
& \fp_5 = (x_{3,1},\,x_{3,2},\,x_{3,3},\,x_{2,2}x_{1,1}-x_{2,1}x_{1,2},\,x_{2,3}x_{1,1}-x_{2,1}x_{1,3},\,x_{2,3}x_{1,2}-x_{2,2}x_{1,3}),\\
& \fp_6 = (x_{1,2},\,x_{2,2},\,x_{3,1},\,x_{3,2},\,x_{3,3}),\\
& \fp_7 = (x_{1,2},\,x_{2,2},\,x_{3,2},\,x_{4,2}),\\
& \fp_8 = (x_{2,1},\,x_{2,2},\,x_{2,3},\,x_{3,2},\,x_{4,2}),\\
& \fp_9 = (x_{3,1},\,x_{3,2},\,x_{4,2},\,x_{2,2}x_{1,1}-x_{2,1}x_{1,2},\,x_{2,3}x_{1,1}-x_{2,1}x_{1,3},\,x_{2,3}x_{1,2}-x_{2,2}x_{1,3}).
\end{aligned}
\]

\noindent We observe that $\mathrm{ht}(\fp_6)=5$ and $\mathrm{ht}(\fp_7)=4$, hence the ideal $I_{\mathrm{adj}}(\cC)$ is not unmixed, and therefore it does not satisfy any of the equivalent conditions in Theorem~\ref{Thm: Unmixed = CI}.
\end{Remark}

\section{On the radicality of adjacent $2$-minor ideals} \label{Sec: Radicality}

In this section, we aim to provide necessary conditions for the radicality of adjacent $2$-minor ideals. We begin with the following consequence of Theorem~\ref{Thm: admissible set}, together with the fact that the radical of an ideal is the intersection of its minimal prime ideals.

    \begin{Corollary}\label{Coro: primary decomposition}
    Let $\cC$ be a collection of cells. Then
    \[
    \sqrt{I_{\mathrm{adj}}(\mathcal{C})} = \bigcap_{W \subseteq V(\mathcal{C})} P_W(\mathcal{C}),
    \]
    where the intersection is taken over all admissible sets $W \subseteq V(\mathcal{C})$.
    \end{Corollary}

The above result yields a primary decomposition of $ \sqrt{I_{\mathrm{adj}}(\mathcal{C})}$, although it is not necessarily minimal. Indeed, not every admissible set of $\mathcal{C}$ corresponds to a minimal prime of $I_{\mathrm{adj}}(\mathcal{C})$, as illustrated in Remark~\ref{ex:notmin}.

Given a collection of cells $\mathcal{C}$, our strategy is to identify subcollections of $\mathcal{C}$ whose presence guarantees the non-radicality of $I_{\mathrm{adj}}(\mathcal{C})$.

     \begin{Definition}\rm\label{Defn:minimally non-radical}\rm
         A collection of cells $\mathcal{C}$ is said to be \textit{minimally non-radical} if $ I_{\mathrm{adj}}(\cC)$ is not radical and, for every cell $C$ of $\mathcal{C}$ such that $\mathcal{C} \setminus \{C\}$ is weakly connected, the adjacent $2$-minor ideal of $\mathcal{C} \setminus \{C\}$ is radical. 
     \end{Definition}

The computational setup and scripts described in the following remark allow us to
visualize and identify collections of cells that are minimally non-radical and to
derive necessary conditions for radicality.

\begin{Remark}\rm \label{Rmk: scripts}
 We have implemented scripts which are publicly available in~\cite{N} and are briefly described below.
 
\smallskip
\noindent
\textup{(1)} Using \textit{SageMath} \cite{sage}, we provide a routine that, for a fixed rank $n$, computes the set $L$ of all collections of cells of rank $n$. The resulting datasets consist of plain-text files listing all tested collections of cells. Each collection is encoded as a list of lists, where each cell is represented by the pair of lists corresponding to its diagonal corners. For instance, a square tetromino with vertex set $[(1,1),(3,3)]$ is encoded as
\[
\texttt{\{\{\{1,1\},\{2,2\}\},\{\{2,1\},\{3,2\}\},\{\{1,2\},\{2,3\}\},\{\{2,2\},\{3,3\}\}\}}.
\]

\noindent
The implementation also includes a routine for filtering out collections that contain
prescribed subconfigurations.

\smallskip
\noindent
\textup{(2)} In \textit{Macaulay2} \cite{M2}, we provide a concise procedure that defines the adjacent $2$-minor ideal of a collection of cells (recently incorporated into the \textit{Macaulay2} package \textit{PolyominoIdeals}~\cite{CNJ} under the name
\href{https://www.macaulay2.com/doc/Macaulay2/share/doc/Macaulay2/PolyominoIdeals/html/_adjacent2__Minor__Ideal.html}{\texttt{adjacent2MinorIdeal}}) and filters out those collections whose associated ideal is non-radical.

\smallskip
\noindent
\textup{(3)} Finally, we include a \textit{Python} script for visualizing the resulting collections of cells, thereby providing a geometric interpretation of the configurations.
\end{Remark}

  The following result introduces an infinite family of minimally non-radical collections of cells.

    \begin{Proposition}\label{Propostion: minimal non-radical}
      Let $\mathcal{D}_t$ be the collection of cells illustrated in Figure~\ref{Figure: Lt}, with $t \geq 2$.
\begin{figure}[h]
	\centering
	\includegraphics[scale=0.9]{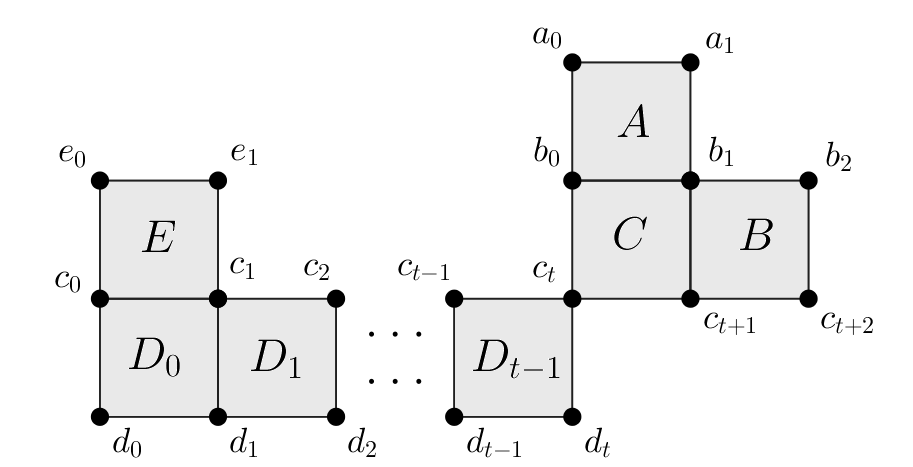}
	\caption{$\mathcal{D}_t$}
	\label{Figure: Lt}
\end{figure}
Then the following statements hold:
\begin{enumerate}
	\item[{\em (1)}] Let $<$ be the lexicographic order induced by the following ordering of the variables:
	\[
    e_1 < e_0 < c_t < \dots < c_0 < d_t < \dots < d_0 < c_{t+2} < c_{t+1} < b_2 < b_1 < b_0 < a_1 < a_0.
	\]
	The reduced Gr\"obner basis of $ I_{\mathrm{adj}}(\mathcal{D}_t)$ with respect to $<$ is given by
	\[
		\cG\left( I_{\mathrm{adj}}(\mathcal{D}_t)\right) \cup \left\{x_{a_0}x_{b_2}x_{c_{t+1}} - x_{b_0}x_{a_1}x_{c_{t+2}}, x_{a_1}x_{b_0}^2x_{c_{t+2}}-x_{a_1}x_{b_0}x_{b_2}x_{c_t}  \right\}.
	\]		
	\item[{\em (2)}] The configuration $\mathcal{D}_t$ is minimally non-radical.
    \item[{\em (3)}] The ideal $I_{\mathrm{adj}}(\mathcal{D}_t)$ is a complete intersection.
	\end{enumerate}
    \end{Proposition}

    \begin{proof}
      (1) With respect to $<$, the initial monomials of the adjacent $2$-minors corresponding to the cells $E, D_0, \ldots, D_{t-1}$ are the diagonal terms, whereas the initial monomials of the adjacent $2$-minors corresponding to the cells $A, B,$ and $C$ are the antidiagonal terms. See Figure~\ref{Figure: Lt 1}, where the initial monomials associated with each cell are indicated.

    \begin{figure}[h]
	\centering
	\includegraphics[scale=0.9]{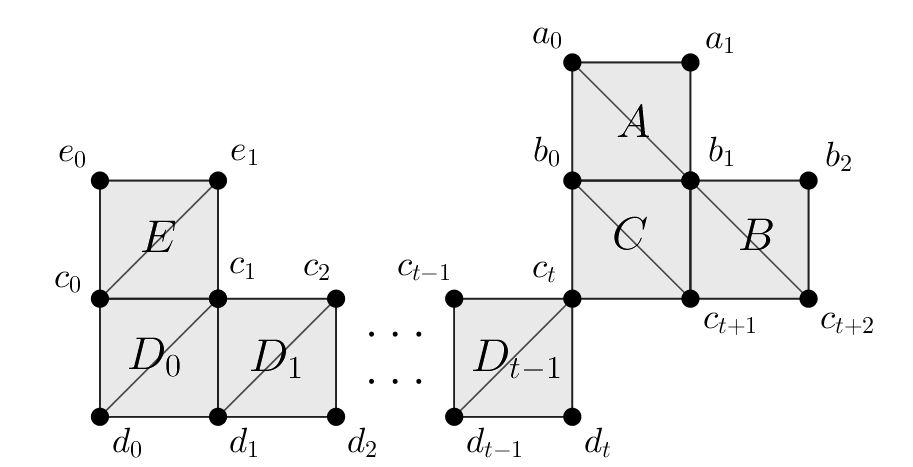}
	\caption{A representation of the leading terms.}
	\label{Figure: Lt 1}
    \end{figure}

      We apply the classical Buchberger’s algorithm to determine the reduced Gr\"obner basis of $I_{\mathrm{adj}}(\mathcal{D}_t)$ with respect to $<$. 

The only non-trivial $S$-polynomial to compute is that of $f_A=  x_{a_0}x_{b_1}-x_{a_1}x_{b_0}$ and $f_B= x_{b_1}x_{c_{t+2}}-x_{b_2} x_{c_{t+1}}$, since $\gcd(\mathrm{in}_{<}(f_A), \mathrm{in}_{<}(f_B)) = x_{b_1}$. In this case, we have 
$$
S(f_A,f_B) = x_{a_0}x_{b_2}x_{c_{t+1}} - x_{b_0}x_{a_1}x_{c_{t+2}} 
\quad \text{and} \quad 
\mathrm{in}_{<}(S(f_A,f_B)) = x_{a_0}x_{b_2}x_{c_{t+1}}.
$$
This binomial cannot be reduced to $0$ modulo $\cG(I_{\mathrm{adj}}(\mathcal{D}_t))$. 
Let $\cG_0 = \mathcal{G}(I_{\mathrm{adj}}(\mathcal{D}_t)) \cup \{ x_{a_0}x_{b_2}x_{c_{t+1}} - x_{b_0}x_{a_1}x_{c_{t+2}} \}$ and set $g = x_{a_0}x_{b_2}x_{c_{t+1}} - x_{b_0}x_{a_1}x_{c_{t+2}}$. 
Now observe that 
$$
S(g,f_A) = x_{a_1}x_{b_0}x_{b_1}x_{c_{t+2}} - x_{a_1}x_{b_0}x_{b_2}x_{c_{t+1}} 
\quad \text{and} \quad 
\mathrm{in}_{<}(S(g,f_A)) = x_{a_1}x_{b_0}x_{b_1}x_{c_{t+2}},
$$
so $S(g,f_A)$ reduces to $0$ modulo $\cG_0$, since $\mathrm{in}_{<}(f_B) = x_{b_1}x_{c_{t+2}}$ and 
$
S(g,f_A) = x_{a_1}x_{b_0}(x_{b_1}x_{c_{t+2}} - x_{b_2}x_{c_{t+1}}).
$
The other case involves 
$$
S(g,f_C) = x_{c_{t+2}}x_{b_0}^2x_{a_1} - x_{a_0}x_{b_1}x_{b_2}x_{c_t} 
\quad \text{where} \quad 
\mathrm{in}_{<}(S(g,f_C)) = x_{a_0}x_{b_1}x_{b_2}x_{c_t}.
$$
A straightforward computation shows that $S(g,f_C)$ reduces to $x_{a_1}x_{b_0}^2x_{c_{t+2}}-x_{a_1}x_{b_0}x_{b_2}x_{c_t}$ modulo $\cG_0$. 
We now set
$$
\cG_1 = \cG_0 \cup \{ x_{a_1}x_{b_0}^2x_{c_{t+2}}-x_{a_1}x_{b_0}x_{b_2}x_{c_t} \},
$$
and denote $h = x_{a_1}x_{b_0}^2x_{c_{t+2}}-x_{a_1}x_{b_0}x_{b_2}x_{c_t}$. 
Routine calculations, which the reader can easily verify, show that the non-trivial $S$-polynomials $S(h,g)$, $S(h,f_B)$, and $S(h,f_C)$ all reduce to $0$ modulo $\cG_1$. 
Therefore, $\cG_1$ is the reduced Gr\"obner basis of $I_{\mathrm{adj}}(\mathcal{D}_t)$ with respect to $<$.
       
\smallskip
\noindent
\textup{(2)}  Let $t \geq2$ and define the following binomial:
 \[
 f_t = x_{a_0} x_{b_2} x_{c_{t+1}} x_{d_1}^2 x_{d_2}\cdots x_{d_{t-1}} x_{e_0} - x_{a_0} x_{b_2} x_{c_{t+1}} x_{d_0} x_{d_1}\cdots x_{d_{t-1}} x_{e_1}.
 \]
Denote by $\cG_1$ the reduced Gr\"obner basis of $ I_{\mathrm{adj}}(\mathcal{D}_t)$ given in (1). From straightforward computations, it follows that $f$ cannot be reduced to $0$ modulo $\cG_1$, whereas $f^2$ reduces to $0$ modulo $\cG_1$. Hence, $ I_{\mathrm{adj}}(\mathcal{D}_t)$ is not radical. Moreover, to obtain a weakly connected collection of cells by removing a cell from $\mathcal{D}_t$, one must remove either $A$, $B$, $E$, or $D_0$. We distinguish the following two cases:

\begin{itemize}
    \item[(a)] Denote by $\cC$ the resulting weakly connected collection obtained by removing either $A$ or $B$ (see Figure \ref{Figure: Propo minimally non-radical} (A)). Observe that the lexicographic order $<$ defined in (1) ensures that $\cG( I_{\mathrm{adj}}(\cC))$ forms the reduced Gr\"obner basis of $ I_{\mathrm{adj}}(\cC)$, since $\mathrm{gcd}(\mathrm{in}_{<}(f),\mathrm{in}_{<}(g)) = 1$ for all $f,g \in \cG( I_{\mathrm{adj}}(\cC))$. Therefore, $I_{\mathrm{adj}}(\cC)$ is radical. 
    \item[(b)] Denote by $\cC'$ the resulting weakly connected collection obtained by removing either $E$ or $D_0$ (see Figure \ref{Figure: Propo minimally non-radical} (B)). Let $<'$ be the lexicographic order induced by the following ordering of the variables:
	\[
   d_t < \dots < d_0 < c_t < \dots < c_0 < e_0 < e_1 < c_{t+1} < c_{t+2} < b_0 < b_1 < b_2 < a_0 < a_1.
	\]
	Then the set of all adjacent 2-minors of $\Cc'$ form the reduced Gr\"obner basis of $ I_{\mathrm{adj}}(\mathcal{D}_t)$ with respect to $<'$. Therefore, $I_{\mathrm{adj}}(\cC')$ is radical.  
\end{itemize}
In conclusion, $\mathcal{D}_t$ is minimally non-radical.

\begin{figure}[h]
 	\centering
 	\subfloat[]{\includegraphics[scale=0.9]{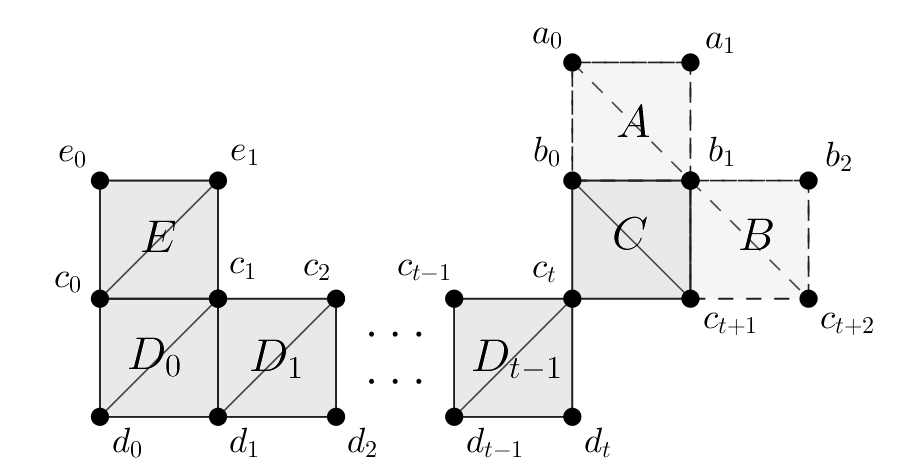}}\qquad\qquad
	\subfloat[]{\includegraphics[scale=0.9]{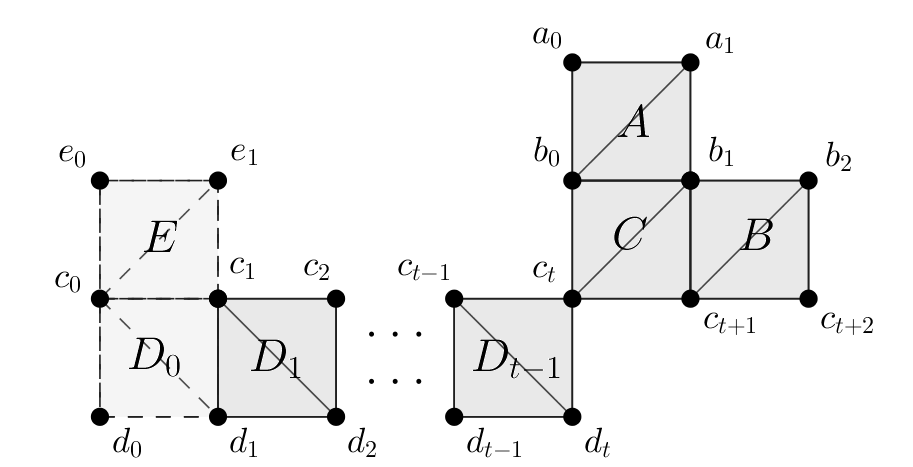}}
    \caption{Cases for the proof of Proposition \ref{Propostion: minimal non-radical}, part (2).}
	\label{Figure: Propo minimally non-radical}
\end{figure}

\smallskip
\noindent
\textup{(3)} We now show that $I_{\mathrm{adj}}(\mathcal{D}_t)$ is complete intersection, that is, $\mathrm{ht}(I_{\mathrm{adj}}(\mathcal{D}_t))=\vert \mathcal{D}_t\vert$. By Proposition \ref{Prop: LC is a minial prime}, we know that $\mathrm{ht}(I_{\mathrm{adj}}(\mathcal{D}_t))\leq \vert \mathcal{D}_t\vert$. Now, let $\mathfrak{p}$ be a minimal prime of $\mathcal{D}_t$. By Theorem \ref{Thm: admissible set} there exists an admissible set of $\mathcal{D}_t$ such that $\mathfrak{p}=P_W(\mathcal{D}_t)$. By Remark~\ref{rem:primePW}, $\mathrm{ht}(\mathfrak{p})=\vert W\vert + \vert (\cD_{t})_{W}\vert$. 
Due to the shape of $\mathcal{D}_t$, it is easy to check that for every admissible set $W$ we have $\vert W\vert + \vert (\cD_{t})_{W}\vert\geq \vert \mathcal{D}_t\vert$, hence $\mathrm{ht}(\mathfrak{p})\geq \vert \mathcal{D}_t\vert$.  
    In conclusion we have $\mathrm{ht}(I_{\mathrm{adj}}(\mathcal{D}_t))=\vert \mathcal{D}_t\vert$, so $I_{\mathrm{adj}}(\mathcal{D}_t)$ is complete intersection.
    \end{proof}

Combining the explicit construction above with computational evidence, we obtain the following result on the existence of minimally non-radical collections of cells by rank.

\begin{Theorem}\label{Thm: infinity minimally non-radical}
Let $\mathfrak{C}(n)$ denote the set of all collections of cells of rank $n\geq 1$. Then:
\begin{enumerate}
    \item For every $\cC \in \mathfrak{C}(1) \cup \mathfrak{C}(2) \cup \mathfrak{C}(3)$, the ideal $I_{\mathrm{adj}}(\cC)$ is radical.
    \item There are no minimally non-radical collections of cells in $\mathfrak{C}(5)$.
    \item For $n = 4$ and for every $n \geq 6$, there exists a minimally non-radical collection of cells in $\mathfrak{C}(n)$.
\end{enumerate}
\end{Theorem}
\begin{proof}
Statements (1) and (2) follow from the computations in \cite{N}. Let us establish statement~(3):
\begin{itemize}
    \item if $n = 4$, it suffices to consider the square tetromino $\cS$ (see Firgure~\ref{fig: square and X-pento}). Indeed, using \textit{Macaulay2} (\cite{CNJ, M2}), one verifies that $I_{\mathrm{adj}}(\cS)$ is not radical and removing any cell from $\cS$ yields the polyomino shown in Figure~\ref{Figure: Example weakly conn. coll. of cells} (B), whose adjacent $2$-minor ideal is radical.
    \item if $n \geq 6$, the claim follows from point~(2) of Proposition~\ref{Propostion: minimal non-radical}. 
    \end{itemize}
This concludes the proof.
\end{proof}

We now turn to necessary conditions for radicality. The following lemma and proposition provides a useful tool for this purpose.

   \begin{Lemma}\label{Lemma: suff cond}
 Let $J$ be an ideal of $K[x_1, \ldots, x_n, y_1, \ldots, y_m]$ such that $J=I+L$, with $I$ and $L$ ideals of $K[x_1, \ldots, x_n]$ and $K[y_1, \ldots, y_m]$, respectively. If $I$ or $L$ is not radical, then $J$ is not radical.
   \end{Lemma}

   \begin{proof}
    If $I$ is not radical, then there exists $f\in K[x_1, \ldots, x_n]$ such that $f\notin I$ but $f^k\in I$ for some $k\geq 2$. Since $I\subset J$, we obtain $f^k\in J$. 
    Suppose that $f \in J=I+L$. Since $f\in K[x_1, \ldots, x_n]$, it follows that $f \notin L$ and $f \in I$, which is a contradiction. Therefore, $f \notin J$, and consequently $J$ is not radical, as claimed. Same argument can be done if $L$ is not radical.
\end{proof}

\begin{Proposition}\label{prop:easy}
 Let $\cC$ and $\cC'$ be two collections of cells such that
$\cC' \subset \cC$. Suppose that for every cell
$A \in \cC \setminus \cC'$, at least one edge of $A$ lies in
$V(\mathcal{C}) \setminus V(\mathcal{C}')$.  If $I_{\mathrm{adj}}(\cC')$ is not radical then $I_{\mathrm{adj}}(\cC)$ is not radical. 
\end{Proposition}

\begin{proof}
  Since  for every cell
$A \in \cC \setminus \cC'$, at least one edge of $A$ lies in
$U=V(\mathcal{C}) \setminus V(\mathcal{C}')$, it yields
\[
I_{\mathrm{adj}}(\mathcal{C})
\subset
I_{\mathrm{adj}}(\mathcal{C}') + (x_u : u \in U).
\]
Then the conclusion follows from Lemma~\ref{Lemma: suff cond}.
\end{proof}

We are now in a position to state a necessary condition for the radicality of adjacent $2$-minor ideals associated with collections of cells.

\begin{Theorem}\label{Prop: necessary conditions}
    Let $\cC$ be a collection of cells. If $I_{\mathrm{adj}}(\cC)$ is radical, then $\cC$ does not contain any of the configurations listed in Table~\ref{tab:nonradical_all}.
\end{Theorem}

\begin{table}[h!]
\centering
\renewcommand{\arraystretch}{1.4}
\setlength{\tabcolsep}{10pt}
\begin{tabular}{|m{1cm}|m{12cm}|}
\hline
\textbf{Rank} & \textbf{Minimally non-radical collections of cells} \\ \hline
\centering 4 &
\includegraphics[scale=0.12]{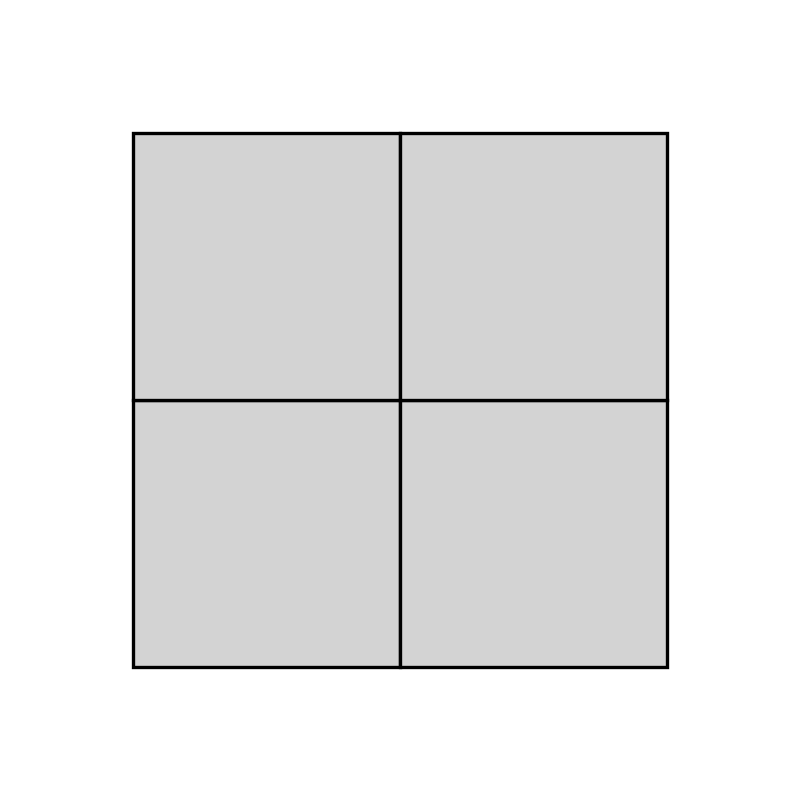} \quad
\includegraphics[scale=0.13]{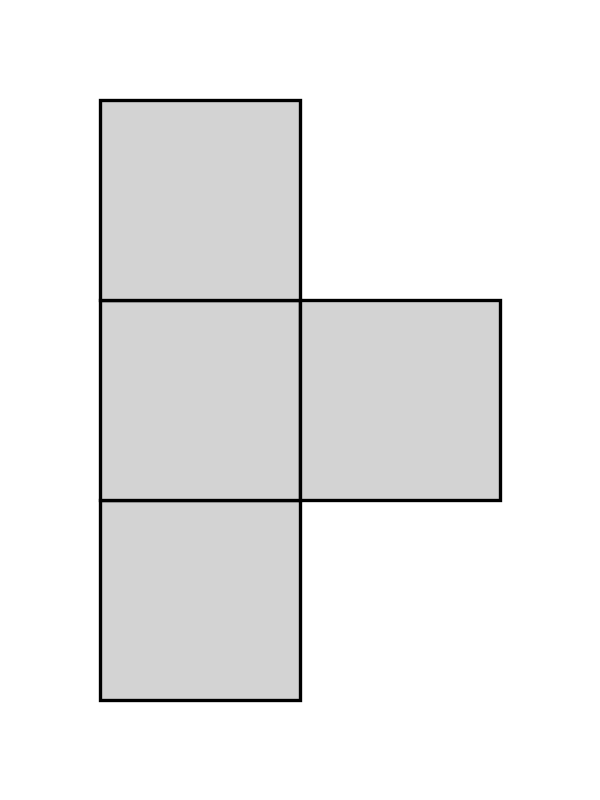} \\ \hline
\centering 6 &
\includegraphics[scale=0.13]{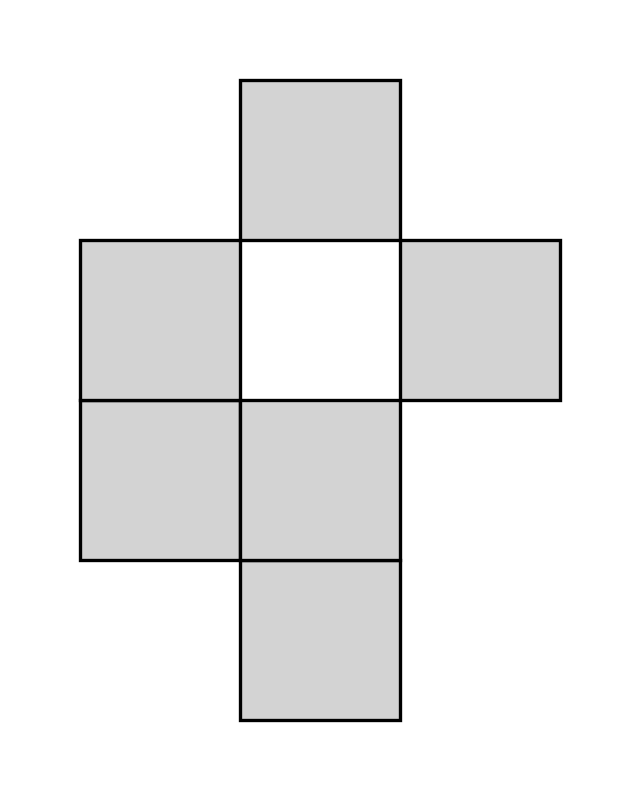} \quad
\includegraphics[scale=0.13]{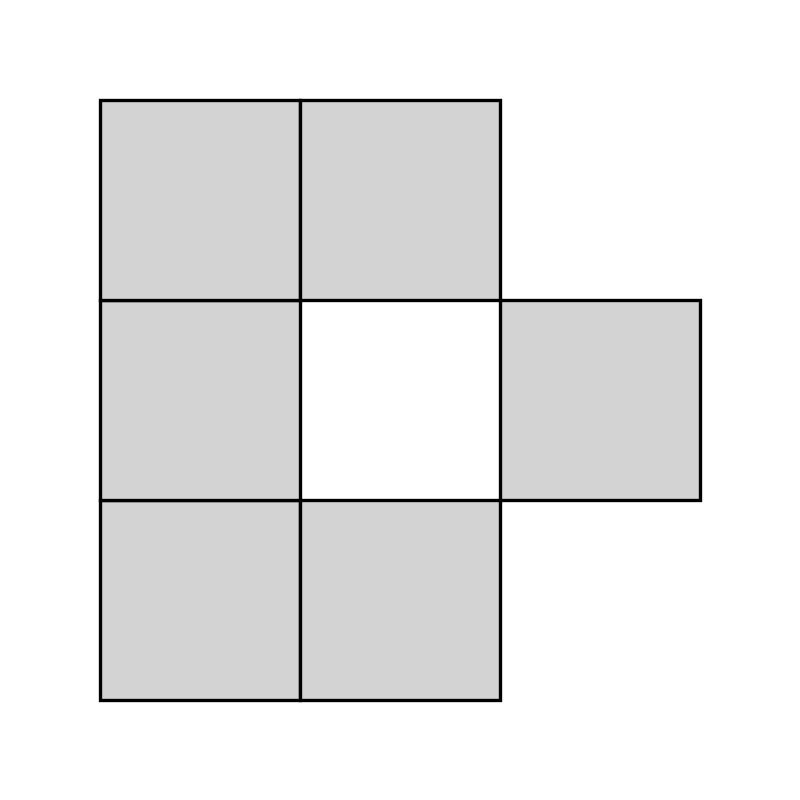} \quad
\includegraphics[scale=0.13]{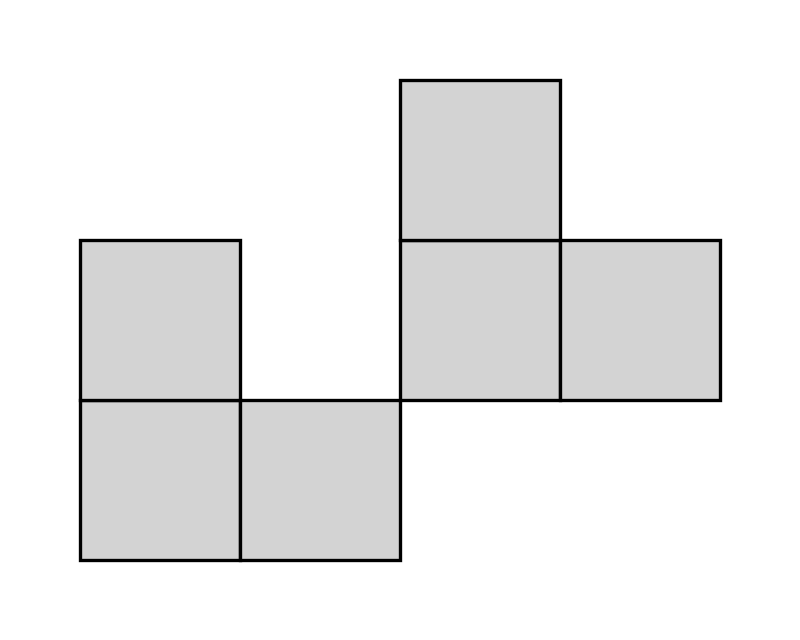} \\ \hline
\centering 7 &
\includegraphics[scale=0.13]{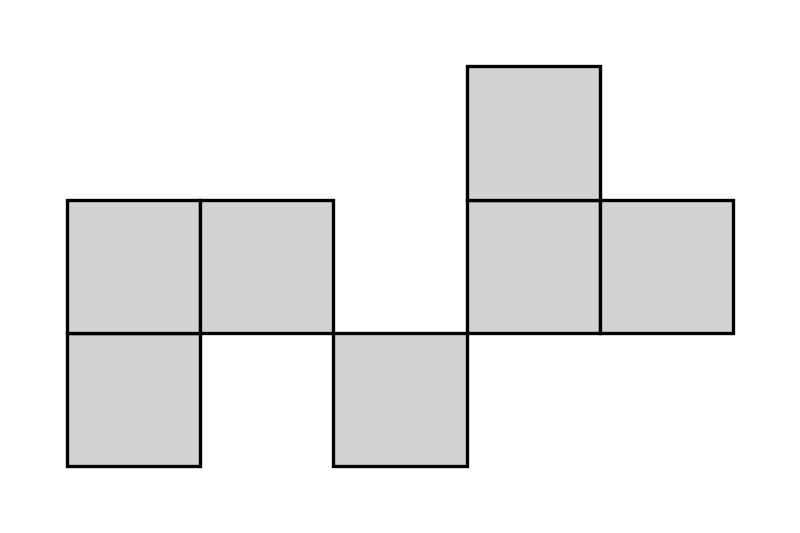} \quad
\includegraphics[scale=0.06]{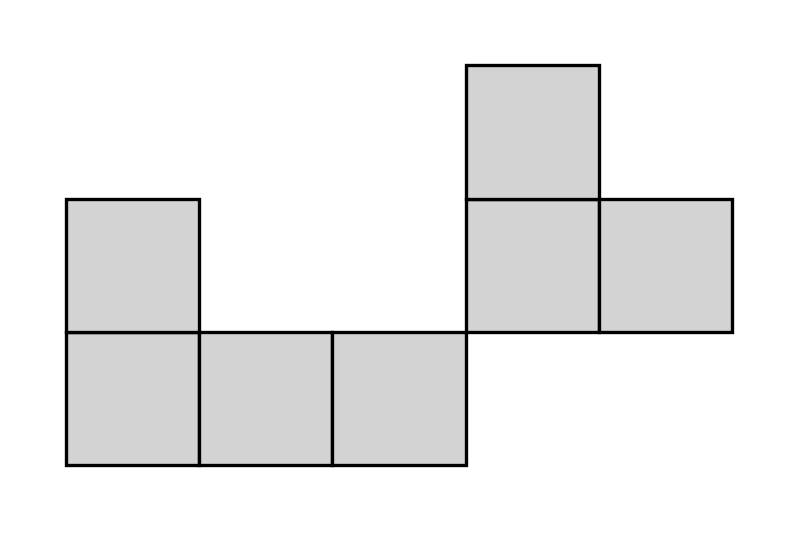} \\ \hline
\centering \vspace{6pt} 8 &
\begin{minipage}[c]{15cm}
\includegraphics[scale=0.16]{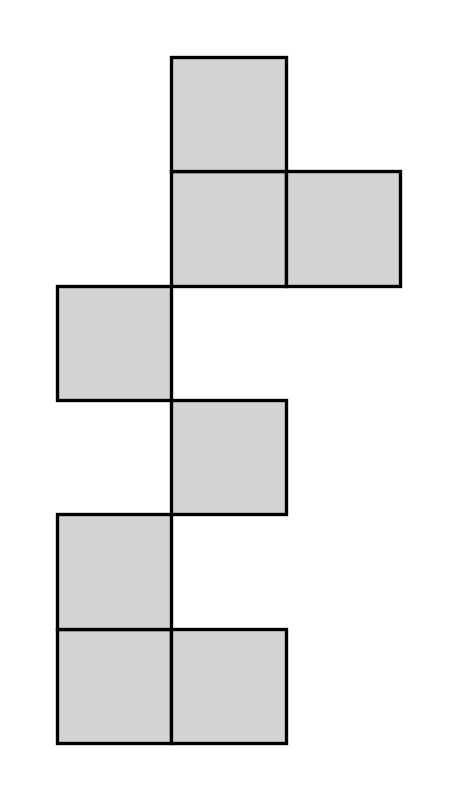} 
\includegraphics[scale=0.16]{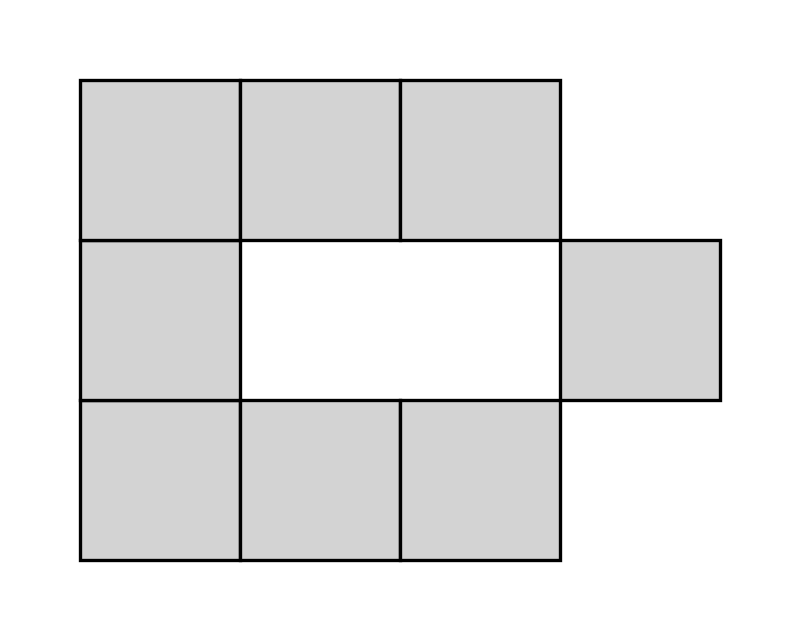} 
\includegraphics[scale=0.16]{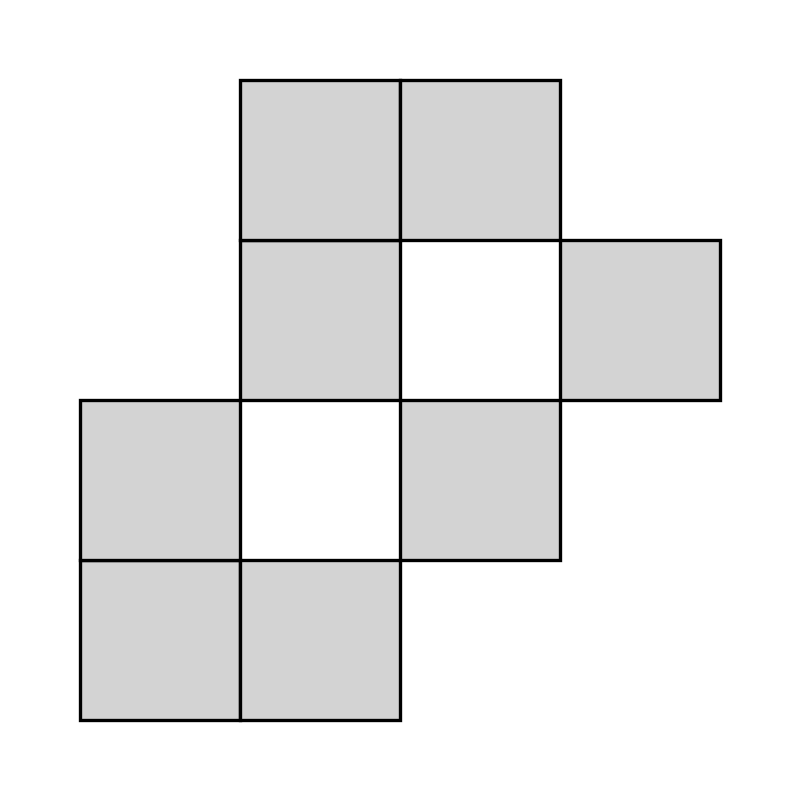} 
\includegraphics[scale=0.16]{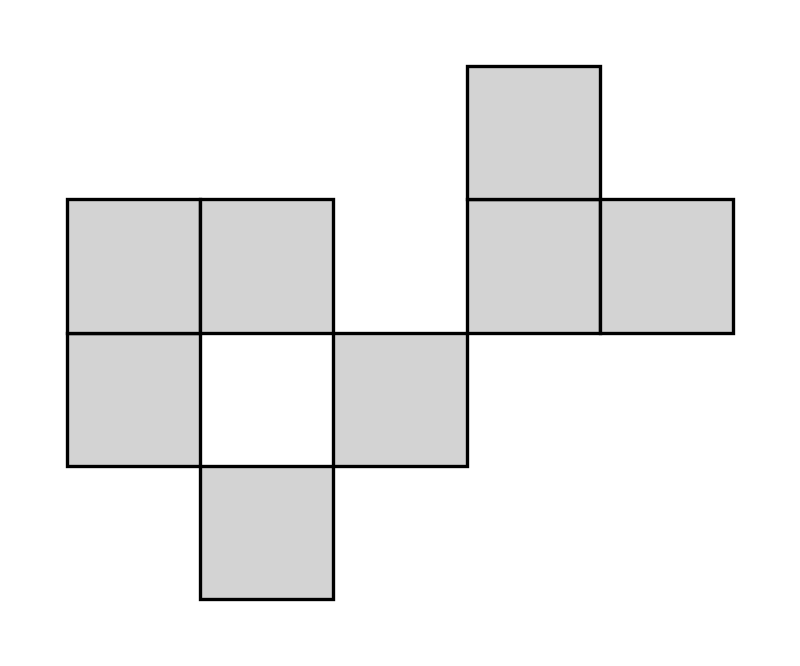}  
\includegraphics[scale=0.16]{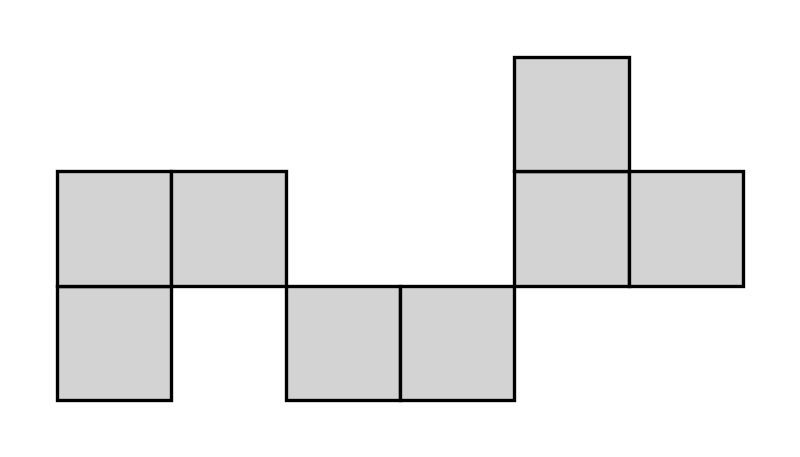} 
\includegraphics[scale=0.16]{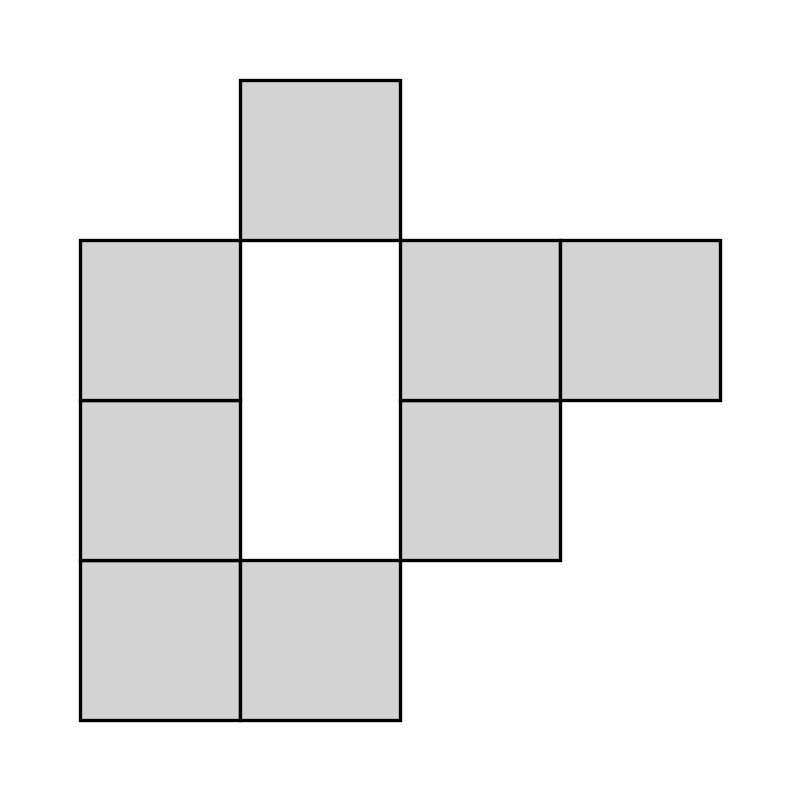} 
\includegraphics[scale=0.16]{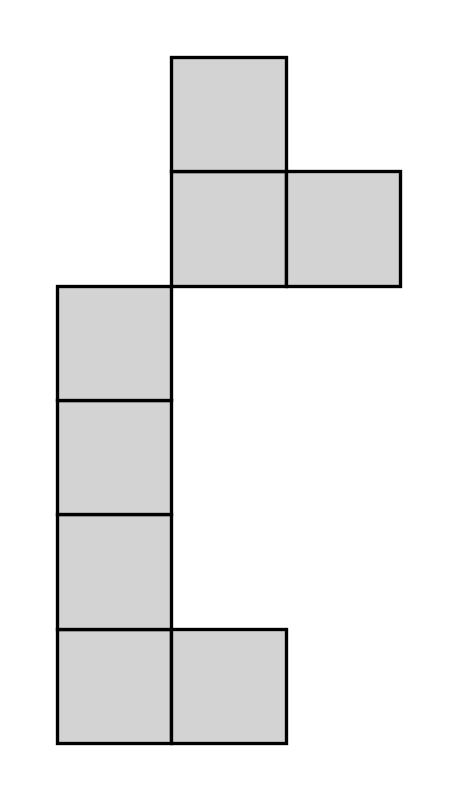} 
\includegraphics[scale=0.16]{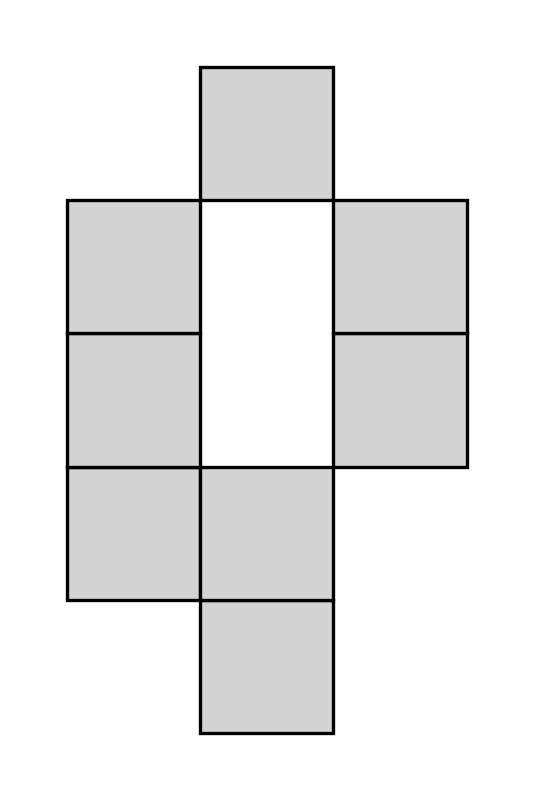}
\end{minipage} \\ \hline
\end{tabular}
\caption{Minimally non-radical configurations of cells according to their rank.}
\label{tab:nonradical_all}
\end{table}

\begin{proof}
Using \textit{Macaulay2}~\cite{CNJ, M2}, we verify that the adjacent $2$-minor ideal associated with each configuration of cells listed in Table~\ref{tab:nonradical_all} is not radical.

Let $\mathcal{C}$ be a collection of cells. We divide the proof according to the type of configuration contained in $\mathcal{C}$, following the ranks listed in Table~\ref{tab:nonradical_all}.

\medskip
\textbf{\textit{Rank 4.}} We first consider the case in which $\cC$ properly contains one of the two configurations of rank four. 

\smallskip
\noindent
\textup{(1)} Assume that $\cC$ properly contains a square tetromino, denoted by $\cS$. Then for every cell
$A \in \cC \setminus \cS$, at least one edge of $A$ lies in $V(\cC) \setminus V(\cS)$. Hence, by Proposition~\ref{prop:easy}, we have that $I_{\mathrm{adj}}(\cC)$ is not radical.

\smallskip
\noindent
\textup{(2)} Now assume that $\cC$ contains the polyomino $\cP$ as in Figure \ref{Figure: no radical coll} (A). If either $A$ or $B$ belongs to $\mathcal{C}$, then $\mathcal{C}$ contains a
square tetromino, and by the previous case,
$I_{\mathrm{adj}}(\mathcal{C})$ is not radical. If neither of these cells is contained in $\cC$, then for every cell
$A \in \cC \setminus \cP$, at least one edge of $A$ lies in $V(\cC) \setminus V(\cP)$. Hence, by Proposition~\ref{prop:easy}, we have that $I_{\mathrm{adj}}(\cC)$ is not radical.

\begin{figure}[h]
	\centering
    \subfloat[$\cP$]{\includegraphics[scale=0.8]{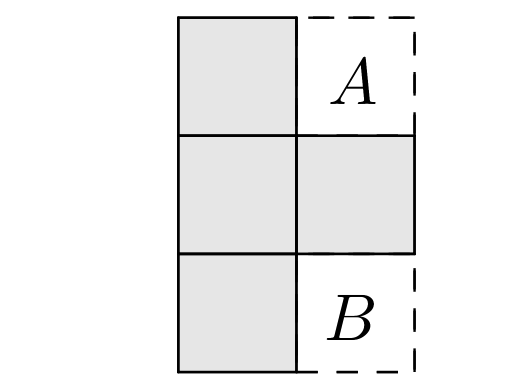}}
    \subfloat[$\cC_1$]{\includegraphics[scale=0.8]{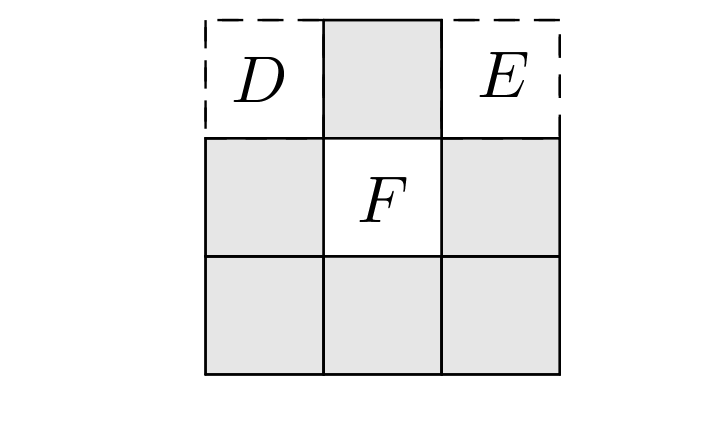}}
    \subfloat[$\cC_2$]{\includegraphics[scale=0.8]{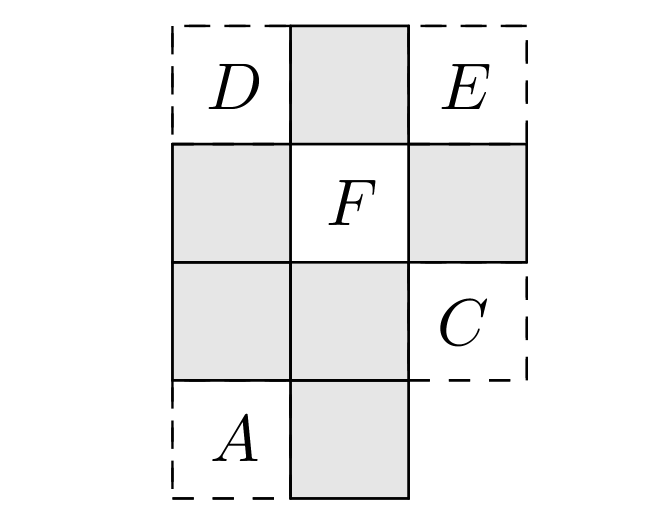}}
    \subfloat[$\cC_3$]{\includegraphics[scale=0.8]{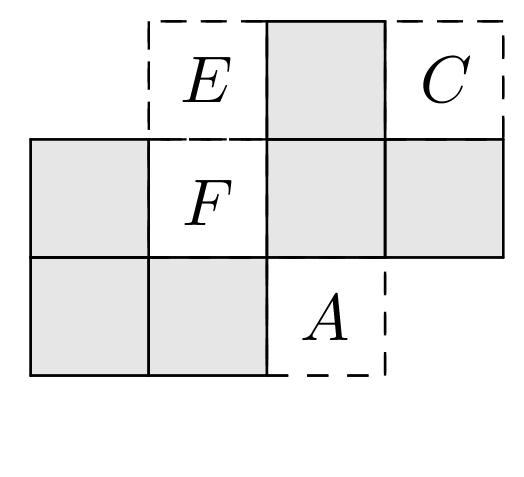}}
	\caption{Some non-radical collections of cells.}
	\label{Figure: no radical coll}
\end{figure}

\medskip
\textbf{\textit{Rank 6.}} We now examine the cases in which $\cC$ contains one of the three configurations of rank six. 

\smallskip
\noindent
\textup{(1)}  Suppose that $\cC$ contains $\cC_1$, shown in Figure \ref{Figure: no radical coll} (A).   If $F \in \cC$, then $\cC$ contains a square tetromino. By the discussion in the Rank~$4$ case, it follows that
$I_{\mathrm{adj}}(\mathcal{C})$ is not radical. 

Now suppose that $F\notin \cC$.  
If $E \in \cC$ and $D \notin \cC$, then using \textit{Macaulay2} (\cite{CNJ, M2}), one verifies that the adjacent $2$-minor ideal of $\cC' = \cC_1 \cup \{E\}$ is not radical. Moreover, for every cell
$A \in \cC \setminus \cC'$, at least one edge of $A$ lies in $V(\cC) \setminus V(\cC')$. Hence, by Proposition~\ref{prop:easy}, $I_{\mathrm{adj}}(\cC)$ is not radical. A similar argument applies when $E \notin \cC$ and $D \in \cC$.

If both cells $E$ and $D$ belong to $\mathcal{C}$, then the adjacent $2$-minor ideal of $\cC'' = \cC_1 \cup \{E,D\}$ is not radical, as verified using \textit{Macaulay2}~\cite{CNJ, M2}. Furthermore, for every cell
$A \in \cC \setminus \cC''$, at least one edge of $A$ lies in $V(\cC) \setminus V(\cC'')$. Hence, Proposition~\ref{prop:easy} again implies that $I_{\mathrm{adj}}(\cC)$ is not radical. 

Finally, if none of the above cells belongs to $\cC$, then for every cell $A \in \cC \setminus \cC_1$, at least one edge of $A$ lies in $V(\cC) \setminus V(\cC_1)$ and the conclusion follows once more from Proposition~\ref{prop:easy}.

\smallskip
\noindent
\textup{(2)} Suppose that $\cC$ contains $\cC_2$, shown in Figure \ref{Figure: no radical coll} (B).  
If either $F$ or $A$ belongs to $\cC$, then $\cC$ contains a square tetromino. By the discussion in Case~Rank~$4$-(1), it follows that $I_{\mathrm{adj}}(\mathcal{C})$ is not radical. 
If $\cC$ contains $C$ or $D$, then $\cC$ also contains the configuration $\cC_1$ discussed above, and the claim follows.  

Now suppose that $\cC$ does not contain $A, C, D,$ and $F$.  
If $E \in \cC$, then using \textit{Macaulay2} (\cite{CNJ, M2}), one verifies that the adjacent $2$-minor ideal of $\cC' = \cC_2 \cup \{E\}$ is not radical. Moreover, for every cell
$A \in \cC \setminus \cC'$, at least one edge of $A$ lies in $V(\cC) \setminus V(\cC')$. Hence, Proposition~\ref{prop:easy} implies that $I_{\mathrm{adj}}(\cC)$ is not radical. 

If none of the above cells belongs to $\cC$, then 
for every cell
$A \in \cC \setminus \cC_2$, at least one edge of $A$ lies in $V(\cC) \setminus V(\cC_2)$ and Proposition~\ref{prop:easy} again yields that
$I_{\mathrm{adj}}(\mathcal{C})$ is not radical.

\smallskip
\noindent
\textup{(3)}  Suppose that $\cC$ contains $\cC_3$, shown in Figure \ref{Figure: no radical coll} (C).  
If either $F$ or $C$ belongs to $\cC$, then $\cC$ contains a square tetromino; similarly, if $A \in \cC$, then $\cC$ contains $\cP$ as in Figure \ref{Figure: no radical coll} (A).
Hence, $I_{\mathrm{adj}}(\cC)$ is not radical following the discussion of Rank 4.  
If $\cC$ contains $E$, then $\cC$ also contains the configuration $\cC_2$ discussed above, and the claim follows. 

Assume that $\cC$ does not contain $A, C, E,$ and $F$. Then for every cell
$A \in \cC \setminus \cC_3$, at least one edge of $A$ lies in $V(\cC) \setminus V(\cC_3)$. Hence, Proposition~\ref{prop:easy} implies that $I_{\mathrm{adj}}(\cC)$ is not radical.  
    
\medskip
\textbf{\textit{Ranks 7 and 8.}}
The remaining cases, corresponding to ranks~$7$ and~$8$, follow by arguments
analogous to those used above.

This concludes the proof.
\end{proof}

We conclude this section by commenting on the scope and limitations of the necessary conditions obtained above, based on computational evidence.

\begin{Remark}\rm \label{Remark: final observations}
Computational evidence shows that the figures presented above represent all minimally non-radical shapes among collections of $n$ cells for $n \leq 8$ (see the outputs in \cite{N}). Providing a complete characterization of collections of cells whose adjacent 2-minor ideals are radical appears to be a highly challenging problem. Indeed, our computations indicate that minimally non-radical shapes emerge with increasing rank in a way that does not exhibit an obvious pattern. For example, when passing from rank $4$ to rank $6$, the three new shapes do not seem to share clear similarities with those of lower rank; from rank $6$ to rank $7$, the first shape of rank $7$ does not appear to derive from any previous configuration; similarly, from rank $7$ to rank $8$, the third, fourth, and sixth shapes appear independently of lower-rank configurations. Naturally, some configurations are ``similar'': for instance, from rank $6$ to rank $9$, the shapes $\cD_2$, $\cD_3$, and $\cD_4$ defined in Proposition \ref{Propostion: minimal non-radical} recur. Moreover, the number of non-radical collections grows rapidly with the rank, making computations increasingly demanding (see Table \ref{tab:collection}).

\begin{table}[h!]
\centering
\begin{tabular}{|>{\centering\arraybackslash}p{6cm}|c|c|c|c|c|c|c|c|c|}
\hline
Rank & 2 & 3 & 4 & 5 & 6 & 7 & 8 \\
\hline
\parbox[c]{6cm}{\centering\vspace{1mm} Number of weakly connected\\collections of cells\vspace{1mm}} & 2 & 5 & 22 & 94 & 524 & 3031 & 18770  \\
\hline
\parbox[c]{6cm}{\centering\vspace{1mm} Number of non-radical\\ collections of cells\vspace{1mm}} & 0 & 0 & 2 & 9 & 74 & 550 & 4210  \\
\hline
\end{tabular}
\caption{Number of weakly connected collections of cells up to symmetries.}
\label{tab:collection}
\end{table}

\end{Remark}

\begin{footnotesize}
{\bf Acknowledgments.} Francesco Navarra and Ayesha Asloob Qureshi were supported by Scientific and Technological Research Council of Turkey T\"UB\.{I}TAK under the Grant No: 124F113. Francesco Navarra is a member of GNSAGA Indam and he acknowledges their support. Sara Saeedi Madani was in part supported by a grant from IPM (No. 1404130019). 

{\bf Declaration of competing interest.}  The authors declare that they have no known competing financial interests or personal relationships that could have appeared to influence the work reported in this paper.\\

{\bf Data availability.} Data used for this article are provided in \cite{N}.
\end{footnotesize}

\end{document}